\documentclass[11pt]{article}
\usepackage{mathrsfs}
\usepackage{amsfonts}
\usepackage{}
\usepackage{amsmath,amssymb,amsthm,latexsym,amstext}
\usepackage[mathscr]{eucal}
\usepackage{rotate,graphics,epsfig,epstopdf}
\usepackage{float}
\usepackage{color}

\newcommand{\be}{\begin{eqnarray}}
\newcommand{\ee}{\end{eqnarray}}
\newcommand{\by}{\begin{eqnarray*}}
\newcommand{\ey}{\end{eqnarray*}}
\newcommand{\bn}{\begin{enumerate}}
\newcommand{\en}{\end{enumerate}}
\newcommand{\bi}{\begin{itemize}}
\newcommand{\ei}{\end{itemize}}

\newtheorem{theorem}{Theorem}[section]
\newtheorem{lemma}{Lemma}[section]

\newtheorem{corollary}{Corollary}[section]
\newtheorem{remark}{Remark}[section]
\newtheorem{defin}{Definition}[section]
\newtheorem{proposition}{Proposition}[section]
\newtheorem{prop}{Proposition}[section]
\newtheorem{example}{Example}[section]

\newtheorem{cor}{Corollary}[section]

\newcommand \bet {\beta}
\newcommand \la {\lambda}
\newcommand \tet {\theta}
\newcommand \sig {\sigma}
\newcommand \kap {\kappa}
\newcommand \alp {\alpha}
\newcommand \gam {\gamma}
\newcommand \eps {\epsilon}
\newcommand \del {\delta}
\newcommand \ome {\omega}

\newcommand \R {\mathbb{R}}
\newcommand \N {\mathbb{N}}
\newcommand \E {\mathbb{E}}
\newcommand \hX {\widehat{X}}
\newcommand \hRj {\widetilde{R}_J}
\newcommand \hRjn {\hRj^{(n)}}
\newcommand \hR {\widetilde{R}}
\newcommand \rhoj {\rho_J}
\newcommand \rhod {\rho_D}

\newcommand \whR {\widehat{R}}
\newcommand \whRn {\whR_n}

\newcommand \mL {\mathcal{L}}
\newcommand \mR {\mathcal{R}}
\newcommand \mF {\mathcal{F}}
\newcommand \fR {\mathfrak{R}}
\newcommand{\mC}{\mathcal{C}}

\newcommand \opsi {\overline{\psi}}
\newcommand \upsi {\underline{\psi}}

\newcommand \lan {\la_n}
\newcommand \kapn {\kap_n}
\newcommand \Fn {F_n}
\newcommand \Rn {R_n}
\newcommand \tetn {\tet_n}
\newcommand \cn {c_n}
\newcommand \Yn {Y_n}
\newcommand \sqrtn {\sqrt{n}}
\newcommand \psin {\psi_n}
\newcommand \mO {\mathcal{O}}
\newcommand \rhojn {\rhoj^{(n)}}
\newcommand \rhodn {\rhod^{(n)}}
\newcommand \Zd {Z_d}
\newcommand \elln {\ell_n}
\newcommand \ynx {y_n(x)}

\newcommand \Sm {S_m}
\newcommand \Smn {S_{m, n}}
\newcommand \xmn {x_{m, n}}
\newcommand \ymn {y_{m, n}}
\newcommand \qm {q_{m}}
\newcommand \Amn {A_{m,n}}
\newcommand \Bmn {B_{m,n}}

\newcommand \vs {v^*}

\renewcommand{\theequation}{\arabic{section}.\arabic{equation}}

\numberwithin{equation}{section}

\begin{document}
\date{\today}
\title{Optimal Reinsurance under the Mean-Variance Premium Principle to Minimize the Probability of Ruin}

\author{Xiaoqing Liang%
\thanks{Department of Statistics, School of Sciences, Hebei University of Technology, Tianjin 300401, P.\ R.\ China, liangxiaoqing115@hotmail.com. X.\ Liang thanks the National Natural Science Foundation of China (11701139, 11571189) and the Natural Science Foundation of Hebei Province (A2018202057) for financial support.}
\and Zhibin Liang
\thanks{School of Mathematical Sciences, Nanjing Normal University, Jiangsu 210023, P.\ R.\ China, liangzhibin111@hotmail.com. Z.\ Liang thanks the National Natural Science Foundation of China (11471165) for financial support.}
\and Virginia R. Young%
\thanks{Corresponding author. Department of Mathematics, University of Michigan, Ann Arbor, Michigan, 48109, vryoung@umich.edu. V.\ R.\ Young thanks the Cecil J. and Ethel M. Nesbitt Professorship of Actuarial Mathematics for financial support.}
}

 \maketitle

\begin{abstract}
We consider the problem of minimizing the probability of ruin by purchasing reinsurance whose premium is computed according to the mean-variance premium principle, a combination of the expected-value and variance premium principles.  We derive closed-form expressions of the optimal reinsurance strategy and the corresponding minimum probability of ruin under the diffusion approximation of the classical Cram\'er-Lundberg risk process perturbed by a diffusion.  We find an explicit expression for the reinsurance strategy that maximizes the adjustment coefficient for the classical risk process perturbed by a diffusion.  Also, for this risk process, we use stochastic Perron's method to prove that the minimum probability of ruin is the unique viscosity solution of its Hamilton-Jacobi-Bellman equation with appropriate boundary conditions.  Finally, we prove that, under an appropriate scaling of the classical risk process, the minimum probability of ruin converges to the minimum probability of ruin under the diffusion approximation.

\medskip


\medskip


\medskip


\end{abstract}

\newpage{}

\section{Introduction}

\setcounter{equation}{0}
\renewcommand{\theequation}
{1.\arabic{equation}}

Reinsurance is an important management tool of insurance companies, due to the protection that reinsurance offers against potentially large losses.  Controlling reinsurance under different criteria, such as minimizing the probability of ruin and maximizing expected utility of terminal wealth, is a popular research topic in the actuarial literature.  Researchers use the tools of stochastic control, including the corresponding Hamilton-Jacobi-Bellman (HJB) equation, to analyze these problems.

We consider the problem of minimizing the probability of ruin by purchasing reinsurance whose premium is computed according to the mean-variance premium principle, a combination of the expected-value and variance premium principles.  Moreover, we do not constrain the form of the optimal reinsurance strategy.   Much of the reinsurance literature constrains the insurance company to buy either pure quota-share reinsurance, pure excess-of-loss reinsurance, or a combination of the two; see, for example, Zhang, Zhou, and Guo \cite{ZZG07}, Liang and Guo \cite{LG07} and \cite{LG08}, and Bai, Cai, and Zhou \cite{BCZ13}.

On the other hand, some researchers have found optimal reinsurance strategies for various optimization problems without restricting the form of the reinsurance.  For example, under the criteria of maximizing expected utility of terminal wealth and minimizing the probability of ruin, Zhang, Meng, and Zeng \cite{ZMZ16} investigated an optimal investment and reinsurance problem in which the insurer purchased a general reinsurance policy and reinsurance is priced according to the mean-variance premium principle, as in this paper.  Liang and Young \cite{LY18} computed the optimal investment and per-loss reinsurance strategies for an insurance company facing a compound Poisson claim process; they assumed that the reinsurer used an expected-value premium principle and showed that the optimal form of reinsurance is excess of loss.

Due to its historical importance in insurance economics and due to its mathematical simplicity, the expected-value principle is often used as a reinsurance premium principle; see, for example, Liang and Young \cite{LY18}, Han, Liang, and Zhang \cite{HLZ2019}, and the references therein.  Although the variance principle is another important premium principle, it is not often used in a dynamic setting; for exceptions, see Hipp and Taksar \cite{HT10}, Zhou and Yuen \cite{ZY12}, Liang and Yuen \cite{LY16}, and Han, Liang, and Yuen \cite{HLY18}.

The mean-variance premium principle combines the expected-value and variance premium principles; therefore, it is more general than either and includes each as a special case.  Under the mean-variance premium principle, Zhang, Meng, and Zeng \cite{ZMZ16} studied optimal investment and reinsurance problems, Chen, Yang, and Zeng \cite{CYZ18} studied a stochastic differential game between two insurers who invest in a financial market and adopt reinsurance to manage their claim risks, and Han, Liang, and Young \cite{HLY19} determined the optimal reinsurance strategy to minimize the probability of drawdown.

Throughout our paper, we consider the problem of minimizing the probability of ruin by purchasing reinsurance whose premium is computed according to the mean-variance premium principle.  We begin by explicitly solving the ruin-minimization problem for the diffusion approximation and by finding the optimal reinsurance to maximize the adjustment coefficient in the classical Cram\'er-Lundberg risk process perturbed by a diffusion.  Then, we relate the minimum probability of ruin under the perturbed classical risk process to the minimum probability of ruin under the corresponding diffusion approximation.  To relate these two probabilities, we scale the classical model by $n > 0$.  Specifically, we multiply the Poisson rate by $n$, divide the claim severity by $\sqrtn$, and adjust the premium rate so that net premium income remains constant.  Iglehart \cite{I1969} introduced the scaled system in queuing theory; Ba\" uerle \cite{BN2004}, Gerber, Shiu, and Smith \cite{GSS2008}, and Cohen and Young \cite{CY2019} used the scaled system to study questions of interest in insurance.

The remainder of the paper is organized as follows.  In Section \ref{sec:2}, we present the model and the ruin-minimization problem.  In Section \ref{sec:3}, under the mean-variance premium principle, we derive explicit expressions for the optimal reinsurance strategy and the corresponding minimum probability of ruin for the diffusion approximation.  We show that the optimal reinsurance strategy also maximizes the adjustment coefficient of the diffusion approximation.  Then, in Section \ref{sec:4}, we find an expression for the reinsurance strategy that maximizes the adjustment coefficient for the classical risk process perturbed by a diffusion, and we prove some interesting properties of that reinsurance strategy.  The end of Section \ref{sec:4} contains one of our main results, namely, that, under the scaling of the classical risk process, the minimum probability of ruin converges to the minimum probability of ruin under the diffusion approximation.  Finally, in Appendix \ref{app:A}, we show that the minimum probability of ruin in the classical risk model perturbed by a diffusion is the unique viscosity solution of its boundary-value problem, which consists of a Hamilton-Jacobi-Bellman equation and boundary conditions.

\section{Model and problem formulation}\label{sec:2}

\setcounter{equation}{0}
\renewcommand{\theequation}
{2.\arabic{equation}}

In this section, we describe the reinsurance market available to the insurance company, and we formulate the problem of minimizing the probability of ruin.  Assume that all random processes exist on the filtered probability space $\big( \Omega, \mF, \mathbb{F} = \{ \mF_t \}_{t \ge 0}, \mathbb{P} \big)$.

Consider an insurer whose surplus process without reinsurance is described by the classical Cram\'{e}r-Lundberg model that is additionally perturbed by a Brownian motion; see Dufresne and Gerber \cite{DG1991} for early work with this model:
\begin{equation}\label{eq:sp}
X^0_t = x + ct - \sum_{i=1}^{N_t} Y_i + \bet W_t,
\end{equation}
in which $X^0_0 = x \ge 0$ is the initial surplus, $c$ is the premium rate,  $\{N_t \}_{t \ge 0}$ is a homogeneous Poisson process with intensity $\la > 0$, $Y_i$ represents the size of the $i$th claim, and the claim sizes $Y_1, Y_2, \ldots$ are independent and identically distributed, positive random variables, independent of $\{N_t\}$. Here, we assume that $F_Y(y)$, the common cumulative distribution function of $\{Y_i\}_{i \in \N}$, is such that $F_Y(0) = 0$,  and $0 < F_Y(y) < 1$ for $y > 0$, that is, each $Y_i$ has full support on $\R^+$.  We assume that the moment generating function of $Y$ exists in a neighborhood of $0$.  Denote the first-order moment by $\E(Y_i) = \mu$ and the second-order moment by $\E\big(Y^2_i\big) = \sig^2$.  We also assume that the premium rate $c$ satisfies $c > \la \mu$.

The term $\sum_{i=1}^{N_t}Y_i$ follows a compound Poisson process (CPP), which represents the aggregate claims up to time $t$. Moreover, $\bet > 0$ is a constant, and $\{W_t\}_{t \ge 0}$ is a standard Brownian motion, independent of the claim number process $\{ N_t\}$ and of the claim severity process $\{Y_i\}_{i \in \N}$. The diffusion term $\bet W_t$ represents the additional uncertainty associated with the insurance market or the economic environment.  This additional uncertainty is not necessarily related to the claims, and we assume that $\bet W_t$ is not reinsurable.

 We assume that the insurer can buy per-loss reinsurance, with a continuously payable premium computed according to the so-called \textit{mean-variance} premium principle, which combines the expected-value and variance premium principles, with risk loadings $\tet$ and $\eta$, respectively.  Specifically, if $R_t(\omega, y)$ represents the retained claim at time $t \ge 0$, as a function of the (possible) claim $Y = y$ at that time and state of the world $\omega \in \Omega$, then reinsurance indemnifies the insurer by the amount $y - R_t(\omega, y)$ if there is a claim $y$ at time $t \ge 0$ and $\omega \in \Omega$, and the time-$t$ premium rate is given by
\begin{equation}\label{eq:MVpp}
 c(R_t) = (1 + \tet) \la \E \big(Y - R_t \big) + \frac{\eta}{2} \, \la \E \big( (Y - R_t)^2 \big).
\end{equation}
We assume that the premium income $c$ is not sufficient to buy full reinsurance, that is, \[
c < (1 + \tet) \la \mu + \frac{\eta}{2} \, \la \sig^2;
\]
otherwise, the insurer would be able to avoid ruin, and the problem of minimizing the probability of ruin would be trivial.  Note that, if $\tet = 0$, the reinsurance premium in \eqref{eq:MVpp} reduces to the variance premium principle; and, if $\eta = 0$, the reinsurance premium reduces to the expected-value principle.

 A retention strategy $\mR = \{ R_t \}_{t \ge 0}$ is {\it admissible} if (i) for a fixed value of $y \ge 0$, the mapping $(t, \omega) \mapsto R_t(\omega, y)$ is $\mathbb{F}$-predictable, (ii) for a fixed pair $(t, \omega)$, $y \mapsto R_t(\omega, y)$ is $\mathcal{B}(\R^+)$-measurable, in which $\mathcal{B}(\R^+)$ denotes the Borel $\sigma$-algebra on $\R^+$, (iii) $0 \le R_t(\omega, y) \le y$, for all $t \ge 0$ and $\omega \in \Omega$, and (iv) the net premium of the controlled surplus is greater than the expected rate of claim payment, that is,
 \begin{equation}\label{eq:drift_condition}
 c - c(R_t) > \la \E R_t,
 \end{equation}
with probability one, for all $t \ge 0$.  Hald and Schmidli \cite{HS04} and Liang and Guo \cite{LG07}, among others, refer to inequality \eqref{eq:drift_condition} as the {\it net-profit condition}.  Denote the set of admissible strategies by $\fR$.\footnote{In the following, we will omit the dependence of $R_t$ on $\omega$ and $y$.  That said, given a Borel-measurable {\it function} $R = R(y)$ such that $0 \le R(y) \le y$, we define a time-homogeneous process $\mR$ by $R_t(\omega, y) = R(y)$. For such a retention function, we will often emphasize its dependence on the possible claim size $Y = y$.}  The insurer's surplus under an admissible retention strategy $\mR$ follows the dynamics
\begin{align}\label{eq:X}
dX_t& = \big(c - c(R_t) \big)dt - R_t dN_t + \bet dW_t \nonumber \\
&= \Big[- \kap + (1 + \tet) \la \E R_t + \la \eta \E (YR_t) - \frac{\eta}{2} \, \la \E \big(R_t^2 \big) \Big] dt - R_t dN_t + \bet dW_t,
\end{align}
in which $\kap$ is the positive constant defined by
\begin{equation}\label{eq:kap}
\kap = (1 + \tet)\la \mu + \frac{\eta}{2} \, \la \sig^2 - c.
\end{equation}

\begin{example}
Note that $\fR$ is non-empty because $\mR = \{ R_t \}$ such that $R_t(y) = y$ for all $y \ge 0$ is in $\fR$.  More generally, any quota-share reinsurance of the form $R_t(y) = q_t y$ with $q_t \in (q_1, 1]$ for all $t \ge 0$ and $y \ge 0$, such that
\[
q_1 = \dfrac{1}{\eta \sig^2} \left[ \big( \eta \sig^2 + \tet \mu \big) - \sqrt{\big( \eta \sig^2 + \tet \mu \big)^2 - 2 \eta \sig^2 \kap/\la} \; \right],
\]
is in $\fR$.  $($Recall that $\E Y = \mu$ and $\E \big(Y^2 \big) = \sig^2.)$  Furthermore, any stop-loss reinsurance of the form $R_t(y) = \min(M_t, y)$ with $M_t > M_1$ for all $t \ge 0$ and $y \ge 0$, such that $M_1 > 0$ uniquely solves
\[
c - \la \mu = \la \int_{M_1}^\infty \big( \tet + \eta(y - M_1) \big) S_Y(y) dy,
\]
is also in $\fR$.  \qed
\end{example}

Next, define the ruin time $\tau_0$ by
\begin{equation}\label{eq:tau}
\tau_0 = \inf\{ t\ge 0: X_t < 0\},
\end{equation}
which depends on the retention strategy $\{R_t\}$.  Note that because of the uncontrolled diffusion term in \eqref{eq:X}, $\tau_0$ also equals  $\inf\{ t\ge 0: X_t \le 0\}$.  Define the minimum probability of ruin by
\begin{equation}\label{eq:psi}
\psi(x) = \inf_{\mR \in \fR} \mathbb{P} \big(\tau_0 < \infty \mid X_0 = x \big).
\end{equation}

Before we tackle the optimization problem for the classical risk model perturbed by a diffusion in \eqref{eq:X}, we first solve this minimization problem by approximating the jump process in \eqref{eq:X} with a diffusion, which we obtain by matching the first two moments at all times $t \ge 0$, as in Grandell \cite{G91}. Specifically,
\begin{equation}\label{eq:approxR}
R_t dN_t \approx \la \E R_t dt - \sqrt{\la \E \big(R^2_t \big)} \, dB_t,
\end{equation}
in which $\{B_t\}_{t \ge 0}$ is a standard Brownian motion, independent of $\{W_t\}$.  By replacing $R_t dN_t$ in \eqref{eq:X} with the approximation in \eqref{eq:approxR}, we obtain the dynamics
\begin{equation}\label{eq:Xhat}
d \hX_t = \Big[- \kap + \tet \la \E R_t + \eta \la \E(Y R_t) - \frac{\eta}{2} \, \la \E \big(R_t^2 \big) \Big] dt + \sqrt{\la \E \big(R^2_t \big)} \, dB_t + \bet dW_t,
\end{equation}
that is, $\{\hX_t \}_{t \ge 0}$ follows a controlled diffusion.  One can show that \eqref{eq:drift_condition} is equivalent to strict positivity of the drift of the diffusion approximation.

Define the ruin time associated with $\{\hX_t\}$ by
\begin{equation}\label{eq:tauD}
\tau_D = \inf\{ t \ge 0: \hX_t < 0\},
\end{equation}
and define the corresponding minimum probability of ultimate ruin by
\begin{equation}\label{eq:psiD}
\psi_D(x) = \inf_{\mR \in \fR} \mathbb{P} \big(\tau_D < \infty \, \big| \, \hX_0 = x \big).
\end{equation}
In the next section, we find the optimal retention strategy to minimize the probability of ruin under the diffusion approximation.

\section{Diffusion approximation risk model}\label{sec:3}

\setcounter{equation}{0}
\renewcommand{\theequation}
{3.\arabic{equation}}

In this section, we solve the ruin minimization problem for the surplus process \eqref{eq:Xhat}.  By standard verification results (see, for example, Theorem 3.1 in Han, Liang, and Young \cite{HLY19}), if we find a classical solution of the following boundary-value problem (BVP) on $\R^+$, then the minimum probability of ruin $\psi_D$ equals that solution.
\begin{equation}\label{eq:hjbxhat}
\begin{cases}
- \kap v_x + \dfrac{1}{2} \, \bet^2 v_{xx} + \la \inf \limits_{R} \left\{ \left(\tet \E R + \eta \E(YR) - \dfrac{\eta}{2} \, \E \big(R^2 \big) \right) v_x +\dfrac{1}{2} \, \E \big(R^2\big) v_{xx} \right\} = 0, \\
v(0) = 1, \quad \lim \limits_{x \to \infty} v(x) = 0.
\end{cases}
\end{equation}
Because the coefficients in the HJB equation in \eqref{eq:hjbxhat} do not depend on the value of the surplus and time, we hypothesize that the optimal reinsurance strategy is state-independent and time-homogeneous.  That is, we assume that the optimal retention strategy $\{ R_t \}_{t \ge 0}$ is defined via a function $R = R(y)$ such that $R_t \equiv R$ for all $t \ge 0$.\footnote{Here, we slightly abuse notation by using $R_t$ on the left side to denote the value of the retention {\it strategy} at time $t$ and by using $R$ on the right side to denote the retention {\it function}.}

For a fixed retention function $R$, such that its corresponding state-independent and time-homogene-ous retention strategy is admissible, the solution of \eqref{eq:hjbxhat} (without the infimum over retention functions) equals the probability of ruin under the retention strategy determined by $R$; direct substitution into \eqref{eq:hjbxhat} without the infimum shows that the solution is given by
\begin{equation}\label{eq:adjcoeff}
\mathbb{P}\big(\tau_D < \infty \, \big| \, \hX_0 = x \big) = e^{- \rhod(R) \, x},
\end{equation}
in which $\rhod(R)$ equals
\begin{equation}\label{eq:rhoDR}
\rhod(R) = 2 \, \dfrac{- \kap + \la \left(\tet \E R + \eta \E(YR) - \dfrac{\eta}{2} \, \E \big(R^2 \big) \right)}{\la \E \big(R^2\big) + \bet^2} \, .
\end{equation}
The (positive) exponent $\rhod(R)$ is called the {\it adjustment coefficient}.\footnote{Note that the numerator in \eqref{eq:rhoDR} equals the drift of the diffusion approximation, which is strictly positive because of the condition in \eqref{eq:drift_condition}. }  Thus, minimizing the probability of ruin under a time-homogeneous and state-independent retention strategy is equivalent to finding a retention function $R$ to maximize $\rhod(R)$.

In the following lemma, we find the retention function $R_D$ that maximizes $\rhod(R)$.   Then, we show that the probability of ruin in \eqref{eq:adjcoeff} with $R = R_D$ solves the BVP in \eqref{eq:hjbxhat}, which implies that it equals the {\it minimum} probability of ruin $\psi_D$.

\begin{lemma}\label{lem:RstarD}
The retention function $R = R_D(y)$ that maximizes $\rhod(R)$ in \eqref{eq:rhoDR} is given by
\begin{equation}
\label{eq:RstarD}
R_D(y) = \dfrac{\tet + \eta y}{\alp^*} \wedge y,
\end{equation}
in which the constant $\alp^* > \eta$ uniquely solves
\begin{equation}
\label{eq:alpha}
\tet \E R + \eta \E(YR) - \frac{\alp}{2} \, \E \big( R^2 \big) = \frac{\bet^2(\alp - \eta) + 2 \kap}{2 \la} \, ,
\end{equation}
with $R = R_D$ given in \eqref{eq:RstarD}.  Moreover, the maximum value of $\rhod(R)$, denoted by $\rhod$, equals
\begin{equation}\label{eq:rhoD}
\rhod = \alp^* - \eta > 0.
\end{equation}
\end{lemma}

\begin{proof}
To maximize $\rhod(R)$ in \eqref{eq:rhoDR}, first fix $\E \big( R^2 \big) = s^2 \in [0, \sig^2]$, in which $\E \big( Y^2 \big) = \sig^2$.  Then, maximizing $\rhod(R)$, subject to the restrictions $\E \big( R^2 \big) = s^2$ and $0 \le R(y) \le y$, is equivalent to maximizing $\tet \E R + \eta \E(YR)$, with $\E \big( R^2 \big) = s^2$ and $0 \le R(y) \le y$.  To that end, define the Lagrangian $\mL$ by
\[
\mL(R) = \tet \E R + \eta \E(YR) - \frac{\alp}{2} \left( \E \big( R^2 \big) - s^2 \right),
\]
in which $\alp \ge 0$ is the Lagrange multiplier.  By using the cumulative distribution function of $Y$, rewrite $\mL(R)$ as follows:
\[
\mL(R) = \int_0^\infty \left[ \tet R(y) + \eta y R(y) - \frac{\alp}{2} \, R^2(y) \right] dF_Y(y) + \frac{\alp}{2} \, s^2.
\]
From this integral representation of $\mL(R)$, we deduce that we can maximize $\mL(R)$ by maximizing the integrand $y$-by-$y$, subject to $0 \le R(y) \le y$.  As a function of $R(y)$, the integrand is a parabola, so it is maximized by
\[
R_D(y) = \dfrac{\tet + \eta y}{\alp} \wedge y.
\]

Next, we show that, given $s^2 \in [0, \sig^2]$, there exists a unique value of $\alp \ge \eta$ such that
\[
s^2 = \E \left( \left( \dfrac{\tet + \eta Y}{\alp} \right)^2 \wedge Y^2 \right),
\]
or equivalently,
\begin{equation}
\label{eq:s_al}
s^2 = 2 \int_0^{\frac{\tet}{\alp - \eta}} y S_Y(y) dy + \frac{2 \eta}{\alp^2} \int_{\frac{\tet}{\alp - \eta}}^\infty (\tet + \eta y) S_Y(y) dy,
\end{equation}
in which $S_Y = 1 - F_Y$.  It is straightforward to show that the right side of \eqref{eq:s_al} decreases from $\sig^2$ to $0$ as $\alp$ increases from $\eta$ to $\infty$.  It follows that \eqref{eq:s_al} has a unique solution $\alp \ge \eta$.

Thus, we have reduced the infinite-dimensional problem of finding a function $R(y)$ to maximize $\rhod(R)$ in \eqref{eq:rhoDR} to the one-dimensional problem of finding the optimal value of $s^2$, or equivalently, of finding the optimal value of $\alp \ge \eta$ because the above argument shows that there is a one-to-one correspondence between $s^2 \in [0, \sig^2]$ and $\alp \ge \eta$.  Thus, we define the function $f$, which we maximize with respect to $\alp$:
\begin{equation}
\label{eq:ratio2}
f(\alp) = \dfrac{- \kap + \la \left( \tet g_1(\alp) + \eta g_2(\alp) - \frac{\eta}{2} \, g_3(\alp) \right)}{\la g_3(\alp) + \bet^2} \, ,
\end{equation}
in which
\begin{equation}
\label{eq:g1}
g_1(\alp) =  \E R = \int_0^{\frac{\tet}{\alp - \eta}} S_Y(y) dy + \frac{\eta}{\alp} \int_{\frac{\tet}{\alp - \eta}}^\infty S_Y(y) dy,
\end{equation}
\begin{equation}
\label{eq:g2}
g_2(\alp) =  \E(YR) = 2 \int_0^{\frac{\tet}{\alp - \eta}} y S_Y(y) dy + \frac{1}{\alp} \int_{\frac{\tet}{\alp - \eta}}^\infty (\tet + 2 \eta y) S_Y(y) dy ,
\end{equation}
and
\begin{equation}
\label{eq:g3}
g_3(\alp) = \E \big( R^2 \big) =  2 \int_0^{\frac{\tet}{\alp - \eta}} y S_Y(y) dy + \frac{2 \eta}{\alp^2} \int_{\frac{\tet}{\alp - \eta}}^\infty (\tet + \eta y) S_Y(y) dy.
\end{equation}
By differentiating $f$ in \eqref{eq:ratio2} with respect to $\alp$, we obtain
\begin{align*}
\frac{\partial f}{\partial \alp} &\propto \left( \la g_3(\alp) + \bet^2 \right) \left(\tet g'_1(\alp) + \eta g'_2(\alp) - \dfrac{\eta}{2} \, g_3'(\alp) \right) - \left( - \kap + \la \left( \tet g_1(\alp) + \eta g_2(\alp) - \frac{\eta}{2} \, g_3(\alp) \right) \right) g'_3(\alp)  \\
&= \left( \la g_3(\alp) + \bet^2 \right) \dfrac{\alp - \eta}{2} \, g_3'(\alp) - \left( - \kap + \la \left( \tet g_1(\alp) + \eta g_2(\alp) - \frac{\eta}{2} \, g_3(\alp) \right) \right) g'_3(\alp)   \\
&\propto \tet g_1(\alp) + \eta g_2(\alp) - \frac{\alp}{2} \, g_3(\alp) - \frac{\bet^2(\alp - \eta) + 2 \kap}{2 \la} \, ,
\end{align*}
in which the second line follows from $\tet g'_1(\alp) + \eta g'_2(\alp) = \frac{\alp}{2} \, g'_3(\alp)$, and the third line follows from $g'_3(\alp) < 0$.   The symbol $\propto$ in the above expression means positively proportional to but not necessarily proportional up to a constant; this usage of $\propto$ is sufficient for our purposes because we are maximizing $f$ over possible values of $\alp$, so we only care about the sign of $\frac{\partial f}{\partial \alp}$.

Define $G$ by the third line above; specifically,
\[
G(\alp) = \tet g_1(\alp) + \eta g_2(\alp) - \frac{\alp}{2} \, g_3(\alp) - \frac{\bet^2(\alp - \eta) + 2 \kap}{2 \la} \, .
\]
We wish to show that $G$ has a unique zero $\alp^* > \eta$.  To that end, first, consider
\[
G(\eta) = \tet \mu + \eta \sig^2 - \frac{\eta}{2} \, \sig^2 - \frac{\bet^2(\eta - \eta) + 2 \kap}{2 \la} = \frac{c - \la \mu}{\la} \, ,
\]
in which we use the expression for $\kap$ in \eqref{eq:kap}.  Recall that we assume $c > \la \mu$; thus, $G(\eta) > 0$.  Next,
\[
\lim_{\alp \to \infty} G(\alp) = - \, \infty.
\]
Finally,
\begin{align*}
G'(\alp) =  \tet g'_1(\alp) + \eta g'_2(\alp) - \frac{1}{2} \, g_3(\alp) - \frac{\alp}{2} \, g'_3(\alp) - \dfrac{\bet^2}{2 \la} = - \, \frac{1}{2} \, g_3(\alp) - \frac{\bet^2}{2 \la} < 0.
\end{align*}
Thus, $G$ has a unique zero $\alp^* > \eta$, from which it follows that $f$ in \eqref{eq:ratio2} has a unique critical point $\alp^* > \eta$.

Furthermore,
\[
\frac{\partial f(\alp)}{\partial \alp} \bigg|_{\alp = \eta} > 0 \qquad \hbox{and} \qquad \lim_{\alp \to \infty} \frac{\partial f(\alp)}{\partial \alp} < 0,
\]
because $G(\eta) > 0$ and $\lim_{\alp \to \infty} G(\alp) < 0$, which implies that $f(\alp)$ is maximized at $\alp = \alp^*$.  Note that \eqref{eq:alpha} is a restatement of $G(\alp) = 0$, and \eqref{eq:rhoD} follows from \eqref{eq:rhoDR} and \eqref{eq:alpha}; thus, we have proved this lemma.
\end{proof}

In the following theorem, we prove that $e^{-\rhod \, x} = e^{- (\alp^* - \eta)x}$ equals the minimum probability of ruin $\psi_D$ defined in \eqref{eq:psiD}.

\begin{theorem}\label{thm:psiD}
Let the constant $\alp^* > \eta$ be the unique solution of
\begin{equation}\label{eq:alpha2}
c - \la \mu = (\alp - \eta) \left\{ \la \int_0^\infty \left( \dfrac{\tet + \eta y}{\alp} \wedge y \right) S_Y(y) dy + \dfrac{\bet^2}{2} \right\}.
\end{equation}
Then, the optimal retention strategy for the diffusion approximation risk model is time-homogeneous, state-independent, and admissible, and it is determined by the retention function
\begin{equation}\label{eq:RstarD2}
R_D(y) = \dfrac{\tet + \eta y}{\alp^*} \wedge y,
\end{equation}
and the minimum probability of ruin equals
\begin{equation}\label{eq:psiD2}
\psi_D(x) =  e^{- (\alp^* - \eta)x} = e^{-\rhod x},
\end{equation}
for $x \ge 0$.
\end{theorem}

\begin{proof}
Equation \eqref{eq:alpha2} that defines $\alp^*$ is the same as equation \eqref{eq:alpha}, after substituting for $\kap$, $\E R$, $\E (YR)$, and $\E \big(R^2 \big)$ from \eqref{eq:kap}, \eqref{eq:g1}, \eqref{eq:g2}, and \eqref{eq:g3}, respectively, and after simplifying the result.   Also, it is easy to see that the retention strategy defined by $R_D$ satisfies conditions (i), (ii), and (iii) of admissibility, so we need to show that it satisfies condition (iv), or equivalently,
\[
- \kap + \tet \la \E \big(R_D\big) + \eta \la \E \big(Y R_D \big) - \dfrac{\eta}{2} \, \la \E \big(R_D^2 \big) > 0.
\]
After using \eqref{eq:alpha} to substitute for the first three terms on the left side of this inequality, we obtain the following equivalent inequality:
\[
\dfrac{\alp^* - \eta}{2} \Big( \bet^2 + \la \E \big(R^2_D \big) \Big) > 0,
\]
which is true because $\alp^* > \eta$ and $\bet > 0$.

It remains for us to show that $\psi_D$ in \eqref{eq:psiD2} solves the BVP in \eqref{eq:hjbxhat} with minimizer $R_D$.  Clearly, $\psi_D$ satisfies the boundary conditions in \eqref{eq:hjbxhat}.  If we substitute $\psi_D(x) = e^{-(\alp^* - \eta)x}$ into the differential equation, we obtain
\[
\kap + \dfrac{\bet^2}{2} \left( \alp^* - \eta \right) + \la \inf_R \left\{ - \left(\tet \E R + \eta \E(YR) - \frac{\eta}{2} \, \E \big(R^2 \big) \right) +\frac{\alp^* - \eta}{2} \, \E \big(R^2\big) \right\} = 0.
\]
It is enough to show that $R_D$ minimizes
\[
\frac{\alp^* - \eta}{2} \, \E \big(R^2\big) - \left(\tet \E R + \eta \E(YR) - \frac{\eta}{2} \, \E \big(R^2 \big) \right).
\]
To that end, fix $\E \big(R^2 \big) = s^2 \in [0, \sig^2]$, as in the proof of Lemma \ref{lem:RstarD}; then, we obtain a minimizer $\tilde R$ of the same form as the maximizer in that lemma, that is,
\[
\tilde R(y) = \dfrac{\tet + \eta y}{\tilde \alp} \wedge y,
\]
for some $\tilde \alp > 0$.  Define the function $h$ by
\[
h(\tilde \alp) = \frac{\alp^* - \eta}{2} \, g_3(\tilde \alp) - \left(\tet g_1(\tilde \alp) + \eta g_2(\tilde \alp) - \frac{\eta}{2} \, g_3(\tilde \alp) \right).
\]
Then,
\[
h'(\tilde \alp) = g'_3(\tilde \alp) \, \dfrac{\alp^* - \tilde \alp}{2},
\]
which is negative for $\tilde \alp < \alp^*$ and positive for $\tilde \alp > \alp^*$.  (Recall that $g'_3(\tilde \alp) < 0$.)  Thus, the minimizer of $h$ equals $\alp^*$, and we are done.
\end{proof}

\begin{remark}
If we set $\bet = 0$, then Theorem {\rm \ref{thm:psiD}} is essentially a special case of Theorem $5.3$ of Zhang, Meng, and Zeng {\rm \cite{ZMZ16}}.  Besides controlling for reinsurance, Zhang, Meng, and Zeng {\rm \cite{ZMZ16}} also control investment in a risky financial market.  But, our proof differs from theirs; indeed, Zhang, Meng, and Zeng {\rm \cite{ZMZ16}} obtain the optimal reinsurance strategy via a probabilistic argument, and we obtain the optimal reinsurance strategy by solving a problem from calculus of variations, namely, maximizing the adjustment coefficient in \eqref{eq:rhoDR}.  \qed
\end{remark}

\begin{remark}
Note that maximizing the adjustment coefficient in \eqref{eq:rhoDR} is equivalent to maximizing the drift divided by the square of the volatility.  Pestien and Sudderth {\rm \cite{PS85}} show that the optimal strategy maximizing the drift divided by the square of the volatility also minimizes the probability of ruin, and we have confirmed that result for the model in this section. \qed
\end{remark}

In the first corollary of Theorem \ref{thm:psiD}, we observe a desirable property of $R_D(y)$ and $y - R_D(y)$.

\begin{corollary}\label{cor:comono}
Because $\alp^* > \eta$, $R_D$ in \eqref{eq:RstarD2} and $y - R_D$ are non-decreasing functions of $y$.  \qed
\end{corollary}

\begin{remark}\label{rem:comono}
Corollary {\rm \ref{cor:comono}} implies that $R_D(Y)$ and $Y - R_D(Y)$ are {\rm comonotonic} random variables.  That both $R_D$ and $y - R_D(y)$ are non-decreasing with respect to $y$ helps prevent moral hazard.  Indeed, if $R_D(y)$ were decreasing with respect to $y$, then the insurer would have an incentive to {\rm create} additional loss to thereby reduce its retention.  Similarly, if $y - R_D(y)$ were decreasing with respect to $y$, then the insurer would have an incentive to {\rm hide} a portion of its loss to thereby increase its reimbursement or indemnity.   \qed
\end{remark}

In the next two corollaries, we consider the special cases of $\eta = 0$ and $\tet = 0$, respectively.  The proofs are straightforward applications of Theorem \ref{thm:psiD}, so we omit them.  If $\eta = 0$, then the assumptions concerning the premium income rate $c$ become $\la \mu < c < (1 + \tet) \la \mu$.  In this case, the mean-variance premium principle reduces to the expected-value premium principle, and the optimal reinsurance strategy is excess-of-loss insurance with a constant deductible.\footnote{When $\bet = 0$ in our model, one can show that the deductible that solves \eqref{eq:alpha_eta0} in Corollary \ref{cor:eta0D} equals the optimal deductible in Zhou and Cai \cite{ZC2014} when $\del(x) = 0$ in their model.  Specifically, equation \eqref{eq:alpha_eta0} is equivalent to $g(x, m) = 0$ on page 426 of Zhou and Cai \cite{ZC2014}.}

\begin{corollary}\label{cor:eta0D}
If $\eta = 0$, then the optimal retention strategy to minimize the probability of ruin is determined by
\begin{equation}\label{eq:RstarD_eta0}
R_D(y) = \dfrac{\tet}{\alp^*} \wedge y,
\end{equation}
in which $\alp^* > 0$ uniquely solves
\begin{equation}\label{eq:alpha_eta0}
(1 + \tet) \la \mu - c + \dfrac{1}{2} \, \bet^2 \alp = \la \int_0^{\frac{\tet}{\alp}} (\tet - \alp y) S_Y(y) dy.
\end{equation}
\end{corollary}

\medskip

If $\tet = 0$, then the assumptions concerning the premium income rate $c$ become $\la \mu < c < \la \mu + \frac{\eta}{2} \, \la \sig^2$; thus, we can write $c = \la \mu + \frac{\eta_0}{2} \, \la \sig^2$ for some $0 < \eta_0 < \eta$.  We use this $\eta_0$ in the expression for $\alp^*$ in the following corollary.  In this case, the mean-variance premium principle reduces to the variance premium principle, and the optimal reinsurance strategy is proportional insurance with a constant proportion.

\begin{corollary}\label{cor:tet0D}
If $\tet = 0$, then the optimal retention strategy to minimize the probability of ruin is determined by
\begin{equation}\label{eq:RstarD_tet0}
R_D(y) = \dfrac{\eta}{\alp^*} \, y,
\end{equation}
in which
\begin{equation}\label{eq:alpha_tet0}
\alp^* = \dfrac{1}{2 \bet^2} \left\{ \big( \eta \bet^2 - (\eta - \eta_0)\la \sig^2 \big) + \sqrt{ \big( \eta \bet^2 - (\eta - \eta_0)\la \sig^2 \big)^2 + 4 \eta^2 \bet^2 \la \sig^2} \right\}.
\end{equation}
\end{corollary}

\section{Classical risk model perturbed by a diffusion}\label{sec:4}

\setcounter{equation}{0}
\renewcommand{\theequation}
{4.\arabic{equation}}

Returning to the classical risk model perturbed by a diffusion in \eqref{eq:X}, $\psi$ is the unique viscosity solution of the following BVP; see Appendix \ref{app:A} for a proof of this statement:  $v(x) = 1$ for $x < 0$, and for $x \ge 0$,
\begin{equation}\label{eq:hjbx}
\begin{cases}
-\kap v_x + \frac{1}{2} \, \bet^2 v_{xx} + \la \inf \limits_{R} \left\{ \left( (1 + \tet) \E R + \eta \E(YR) - \dfrac{\eta}{2} \, \E \big(R^2 \big) \right) v_x + \E v(x - R) - v(x) \right\} = 0, \\
v(0) = 1, \quad \lim \limits_{x \to \infty} v(x) = 0.
\end{cases}
\end{equation}
Because it is difficult, if not impossible, to obtain an explicit expression for the solution of \eqref{eq:hjbx}, in Section \ref{sec:41}, we obtain the optimal reinsurance strategy to maximize the adjustment coefficient for the classical risk model perturbed by a diffusion; see, for example, Centeno \cite{C86}, Centeno and Sim\~oes \cite{CS91}, Hald and Schmidli \cite{HS04}, Liang and Guo \cite{LG07} and \cite{LG08}, Wei, Liang, and Yuen \cite{WLY18}, and Zhang and Liang \cite{ZL16}.

Then, in Section \ref{sec:43}, we adapt the work in Cohen and Young \cite{CY2019} to modify $\psi_D$ by terms of order $\mathcal{O}\big( \la^{-1/2} \big)$ to show that, for large values of $\la$ and for $Y$, $\tet$, and $c$ scaled appropriately, $\psi_D$ and $\psi$ are approximately equal.

\subsection{Maximizing the adjustment coefficient}\label{sec:41}

In Section \ref{sec:3}, we showed that minimizing the probability of ruin for the diffusion approximation is equivalent to maximizing the adjustment coefficient for that model.  Thus, in this section, we begin by maximizing the adjustment coefficient for the classical risk model perturbed by a diffusion.

Let $\rho(R)$ be the adjustment coefficient for the classical risk model perturbed by a diffusion in \eqref{eq:X} for a given retention function $R$, from which one can define a time- and state-independent retention {\it strategy}.\footnote{Recall that the optimal retention strategy for the diffusion approximation is both time- and state-independent; see the expression in \eqref{eq:RstarD2}.}  Then, $\rho(R) > 0$ satisfies the following equation; compare with equation (7.2) in Dufresne and Gerber \cite{DG1991}:
\begin{equation}\label{eq:rhoJR}
\left[- \kap + \la \left( (1 + \tet) \E R + \eta \E(Y R) - \frac{\eta}{2} \, \E\big(R^2\big) \right) \right] r - \frac{1}{2} \, \bet^2 r^2 - \la \big( M_R(r) - 1 \big) = 0,
\end{equation}
in which $M_R(r) = \E \big( e^{rR} \big)$ is the moment generating function of the random variable $R = R(Y)$ evaluated at $r$.  Recall that we assume $\E \big( e^{rY} \big)$ is finite in a neighborhood of $0$; thus, it follows from $0 \le R(y) \le y$ for all $y \ge 0$ that $\E \big( e^{rR} \big)$ is finite in that same neighborhood of $0$.  Our goal is to maximize $\rho(R)$, that is, to find $\rhoj$ given by
\[
\rhoj = \sup_{R} \rho(R),
\]
and we also want to find the maximizing retention function $R$.  To that end, define the function $j$ by the left side of \eqref{eq:rhoJR}, that is,
\begin{equation}\label{eq:j}
j(r) = \left[- \kap + \la \left( (1 + \tet) \E R + \eta \E(Y R) - \frac{\eta}{2} \E\big(R^2\big) \right) \right] r - \frac{1}{2} \, \bet^2 r^2 - \la \big( M_R(r) - 1 \big),
\end{equation}
for $r \ge 0$.  Note that $j(0) = 0$,
\[
j'(r) = \left[- \kap + \la \left( (1 + \tet) \E R + \eta \E(Y R) - \frac{\eta}{2} \E\big(R^2\big) \right) \right] - \bet^2 r - \la \E \big( R e^{rR} \big),
\]
and
\[
j''(r) = - \bet^2 - \la \E \big( R^2 e^{rR} \big) < 0.
\]
Because $j$ is concave with $j(0) = 0$, one necessary condition for \eqref{eq:rhoJR} to have a positive solution is that $j'(0) > 0$, or equivalently,
\begin{equation}\label{eq:jprime0}
c - \la \E Y > \la \left( \tet \E(Y - R) + \dfrac{\eta}{2} \, \E\big( (Y - R)^2 \big) \right),
\end{equation}
which is equivalent to the net-profit condition in \eqref{eq:drift_condition}.  If the following holds,
\begin{equation}\label{eq:r_infty}
\lim \limits_{r \to (r_\infty)-} j(r) < 0,
\end{equation}
in which $r_\infty = \sup \{ r > 0: M_R(r) < \infty \}$, then there exists a unique solution $\rho(R) > 0$ of \eqref{eq:rhoJR} for a given retention function $R$.

Because $j(r) < 0$ for $\rho(R) < r < r_\infty$, the left side of \eqref{eq:rhoJR} is non-positive at the maximum $\rhoj$ when $R$ is any retention function for which the adjustment coefficient exists, and the left side of \eqref{eq:rhoJR} is exactly equal to $0$ at the optimal $R$, which we denote by $\hRj$.  In other words, $\rhoj$ solves
\[
\sup_{R} \left\{ \left[ - \kap + \la \left( (1 + \tet) \E R + \eta \E(Y R) - \frac{\eta}{2} \, \E\big(R^2\big) \right) \right] r - \frac{1}{2} \, \bet^2 r^2 - \la \big( M_R(r) - 1 \big) \right\} = 0,
\]
or equivalently,
\begin{equation}\label{eq:rhoj}
- \kap r - \frac{1}{2} \bet^2 r^2 + \la \sup_{R} \left\{ \left((1 + \tet) \E R + \eta \E(Y R) - \frac{\eta}{2} \, \E\big(R^2 \big) \right) r  - \big( M_R(r) - 1 \big) \right\} = 0.
\end{equation}
See Guerra and Centeno \cite{GC04} for an interesting paper that relates maximizing the adjustment coefficient to maximizing expected utility of surplus.

In the following theorem, we present the solution $\rhoj > 0$ of \eqref{eq:rhoj} and the corresponding optimal retention function $\hRj$.

\begin{theorem}\label{thm:rhoj}
The maximum adjustment coefficient $\rhoj > 0$ for the risk process in \eqref{eq:X} uniquely solves
\begin{equation}\label{eq:rhoj2}
c - \la \mu = r \left\{ \la \int_0^\infty \dfrac{e^{r \hRj(y)} - 1}{r} \, S_Y(y) dy + \dfrac{\bet^2}{2} \right\},
\end{equation}
and the corresponding optimal retention function $\hRj$ is given by
\begin{equation}\label{eq:hRj}
\hRj(y) =
\begin{cases}
y, &\quad 0 \le y < \dfrac{1}{\rhoj} \ln(1 + \tet), \vspace{1ex} \\
R_c(\rhoj, y), &\quad y \ge \dfrac{1}{\rhoj} \ln(1 + \tet),
\end{cases}
\end{equation}
in which $R_c(r, y) \in [0, y]$ uniquely solves
\begin{equation}\label{eq:Rc}
(1 + \tet) + \eta y - \eta R - e^{r R} = 0,
\end{equation}
for $y \ge \frac{1}{r} \ln(1 + \tet)$ and for any $r > 0$.
\end{theorem}

\begin{proof}
Define the function $k$ by the expression in the curly brackets of \eqref{eq:rhoj}, ignoring the last term $1$.  Specifically,
\begin{align}\label{eq:k}
k(r, R)&= \left((1 + \tet) \E R + \eta \E(Y R) - \frac{\eta}{2} \, \E\big(R^2 \big) \right) r  - M_R(r) \notag  \\
&= \int_0^\infty \left[ \left((1 + \tet) R(y) + \eta y R(y) - \frac{\eta}{2} \, R^2(y) \right) r  - e^{r R(y)} \right] dF_Y(y).
\end{align}
For a given value of $r$, we wish to find $\hR(r, y)$ that maximizes $k$.  Consider the integrand in the second line of the expression for $k$, namely,
\[
\ell(R) = \left((1 + \tet) R + \eta y R - \frac{\eta}{2} \, R^2 \right) r  - e^{r R}.
\]
By differentiating with respect to $R$, we obtain
\[
\ell_R(R) = \big((1 + \tet) + \eta y - \eta R \big) r  - r e^{r R},
\]
and
\[
\ell_{RR}(R) = - \eta r  - r^2 e^{r R} < 0.
\]
Thus, $\ell$ is strictly concave with respect to $R$, and its maximizer, for a given value of $r$ and subject to $0 \le R \le y$, is given by
\[
\hR(r, y) = R_c(r, y) \wedge y,
\]
in which $R_c(r, y) > 0$ solves \eqref{eq:Rc}.  Equation \eqref{eq:Rc} has a unique positive solution $R_c(r, y)$ because the left side decreases from $\tet + \eta y > 0$ to $-\infty$ as $R$ increases from $0$ to $\infty$.  Moreover, by differentiating \eqref{eq:Rc} with respect to $y$, we obtain
\begin{equation}\label{eq:Rc_y}
\dfrac{\partial R_c(r, y)}{\partial y} = \dfrac{\eta}{\eta + r e^{r R_c(r, y)}} \in (0, 1).
\end{equation}
Also, note that if $y$ equals $\frac{1}{r} \ln(1 + \tet)$, then $R_c(r, y)$ also equals $\frac{1}{r} \ln(1 + \tet)$.  Thus, because $R_c$ increases with $y$ but at a rate less than $1$, we can write $\hR$ as follows:
\begin{equation}\label{eq:hR}
\hR(r, y) =
\begin{cases}
y, &\quad 0 \le y < \dfrac{1}{r} \ln(1 + \tet), \vspace{1ex} \\
R_c(r, y), &\quad y \ge \dfrac{1}{r} \ln(1 + \tet).
\end{cases}
\end{equation}
Substitute $\hRj(y) := \hR(\rhoj, y)$ into \eqref{eq:rhoj}, use the expression for $\hR(\rhoj, y)$ in \eqref{eq:hR}, and perform integration by parts and simplify to obtain \eqref{eq:rhoj2}.

It remains to show that \eqref{eq:rhoj2} has a unique solution $\rhoj > 0$, in which the arg max is given by $\hRj$ in \eqref{eq:hR}.  To that end, one can show that \eqref{eq:rhoj2} is equivalent to $G(\rhoj) = 0$, in which $G$ is given by
\[
G(r) = \la \int_0^{\frac{1}{r} \ln(1 + \tet)} \big( 1 + \tet + \eta y - e^{ry} \big) S_Y(y) dy + \la \eta \int_{\frac{1}{r} \ln(1 + \tet)}^\infty R_c(r, y) S_Y(y) dy - \kap - \dfrac{1}{2} \, \bet^2 r.
\]
(Indeed, in \eqref{eq:rhoj2}, (i) substitute for
\[
c - \la \mu = \la \Big(\tet \mu + \frac{\eta}{2} \, \sig^2 \Big) - \kap = \la \int_0^\infty (\tet + \eta y) S_Y(y) dy - \kap,
\]
(ii) substitute for $e^{\hRj(y)}$ in the integrand on the right side using \eqref{eq:hRj} and \eqref{eq:Rc}, and (iii) simplify the result to get $G(\cdot) = 0$.)  By differentiating \eqref{eq:Rc} with respect to $r$, we obtain
\[
\dfrac{\partial R_c}{\partial r} = - \, \dfrac{R_c e^{r R_c}}{\eta + r e^{r R_c}} < 0.
\]
Note that $G(0) = c  - \la \E Y > 0$, $\lim_{r \to \infty} G(r) = - \infty$, and
\[
G'(r) = - \la \int_0^{\frac{1}{r} \ln(1 + \tet)} y e^{ry} S_Y(y) dy + \la \eta \int_{\frac{1}{r} \ln(1 + \tet)}^\infty \dfrac{\partial R_c(r, y)}{\partial r} S_Y(y) dy - \dfrac{1}{2} \, \bet^2 < 0.
\]
Thus, $G$ has a unique zero $\rhoj$.
\end{proof}

\begin{remark}
Note the parallel between the expressions for $\rhod = \alp^* - \eta$ in \eqref{eq:alpha2} and for $\rhoj$ in \eqref{eq:rhoj2}.  They only differ in the integrand on the right side of each equation. In the former, the integrand has a factor of $R_D(y);$ in the latter, $(e^{\rhoj \hRj(y)} - 1)/\rhoj$, whose first-order linear approximation equals $\hRj(y)$.  \qed
\end{remark}

In the first corollary, we prove some interesting properties of $\hRj$.  Compare the first property in this corollary with Corollary \ref{cor:comono}, and see the discussion in Remark \ref{rem:comono}.

\begin{corollary}\label{cor:prop_hRj}
$\hRj$ satisfies the following properties:
\begin{enumerate}
\item{} $\hRj$ increases with rate lying in $(0, 1];$ thus, $\hRj$ in \eqref{eq:hRj} and $y - \hRj$ are non-decreasing functions of $y$.
\item{} $\hRj$ is concave, strictly for $y > \frac{1}{\rhoj} \, \ln (1 + \tet)$.
\item{} $\hRj$ has the curvilinear asymptote $g$ given by
\[
g(y) = \dfrac{1}{\rhoj} \, \ln \big(1 + \tet + \eta y),
\]
and $\hRj(y) \le g(y) \wedge y$ for all $y \ge 0$.
\end{enumerate}
\end{corollary}

\begin{proof}
Equations \eqref{eq:Rc_y} and \eqref{eq:hR} show that $\hRj$ is non-decreasing with respect to $y$ with rate of increase less than or equal to 1; property 1 follows.

By differentiating $\frac{\partial R_c(r, y)}{\partial y}$ with respect to $y$, we see that the second derivative of $R_c(r, y)$ is negative; thus, $\hRj(y)$ is strictly concave for $y > \frac{1}{\rhoj} \, \ln(1 + \tet)$.  $\hRj(y)$ is linear and, hence, concave for $y < \frac{1}{\rhoj} \, \ln(1 + \tet)$.  Moreover,
\[
\dfrac{\partial R_c(r, y)}{\partial y} \Bigg|_{y = \frac{1}{\rhoj} \, \ln(1 + \tet) +} = \dfrac{\eta}{\eta + (1 + \tet) \rhoj} < 1 = \dfrac{\partial R_c(r, y)}{\partial y} \Bigg|_{y = \frac{1}{\rhoj} \, \ln(1 + \tet) -}.
\]
Thus, $\hRj$ is concave across $y = \frac{1}{\rhoj} \, \ln(1 + \tet)$, and we have shown property 2.

Finally, it is clear that $\hRj(y) \le g(y) \wedge y$ for all $y \ge 0$, and a straightforward calculation shows that
\[
\lim \limits_{y \to \infty} \left( \hRj(y) - \dfrac{1}{\rhoj} \, \ln \big(1 + \tet + \eta y) \right) = 0.
\]
We have, thereby, shown property 3.
\end{proof}

In the following two corollaries, we consider the special cases of $\eta = 0$ and $\tet = 0$, respectively.  In the first corollary, when $\eta = 0$, the mean-variance premium principle reduces to the  expected-value premium principle, and the optimal reinsurance strategy is excess-of-loss insurance with a constant deductible, as in Corollary \ref{cor:eta0D}.

\begin{corollary}
If $\eta = 0$, then the optimal retention strategy to maximize the adjustment coefficient is determined by
\begin{equation}\label{eq:hRj_eta0}
\hRj(y) = \dfrac{1}{\rhoj} \ln(1 + \tet) \wedge y,
\end{equation}
for all $y \ge 0$, in which the maximum adjustment coefficient $\rhoj > 0$ solves
\begin{equation}\label{eq:rhoj_eta0}
(1 + \tet) \la \mu - c + \dfrac{1}{2} \, \bet^2 r = \la \int_0^{\frac{1}{r} \ln(1 + \tet)} \big( 1 + \tet - e^{ry} \big) S_Y(y) dy.
\end{equation}
Moreover, because $\frac{\ln(1 + \tet)}{\rhoj} \le \frac{\tet}{\rhod}$ when $\eta = 0$, it follows that $\hRj(y) \le R_D(y)$ for all $y \ge 0$.
\end{corollary}

\begin{proof}
Proving \eqref{eq:hRj_eta0} and \eqref{eq:rhoj_eta0} is a straightforward application of Theorem \ref{thm:rhoj}, so we omit those details.  Now, $\hRj(y) \le R_D(y)$ for all $y \ge 0$ if and only if
\[
\dfrac{1}{\rhoj} \ln(1 + \tet) \le \dfrac{\tet}{\alp^*} \, .
\]
From \eqref{eq:rhoD}, we know that $\alp^* = \rhod$ when $\eta = 0$; thus, this inequality is equivalent to
\[
\dfrac{1}{\rhoj} \ln(1 + \tet) \le \dfrac{\tet}{\rhod} \, ,
\]
or
\[
d_J \le d_D,
\]
in which we define $d_J = \frac{1}{\rhoj} \ln(1 + \tet)$ and $d_D = \frac{\tet}{\rhod}$.

From \eqref{eq:rhoj_eta0}, we deduce that $d_J$ solves
\begin{equation}\label{eq:dJ}
\la \int_0^d \left( (1 + \tet) - (1 + \tet)^{y/d} \right) S_Y(y) dy - \dfrac{1}{2} \, \frac{\bet^2 \ln(1 + \tet)}{d} = (1 + \tet) \la \mu - c.
\end{equation}
Similarly, from \eqref{eq:alpha_eta0}, we deduce that $d_D$ solves
\begin{equation}\label{eq:dD}
\tet \la \int_0^d \left( 1 - \dfrac{y}{d} \right) S_Y(y) dy - \dfrac{1}{2} \, \frac{\bet^2 \tet}{d} = (1 + \tet) \la \mu - c.
\end{equation}
The left side of \eqref{eq:dD} increases with respect to $d$; thus, to show that $d_J \le d_D$, it is enough to show that the left side of \eqref{eq:dD} is less than the right side when we set $d = d_J$, that is,
\[
\tet \la \int_0^{d_J} \left( 1 - \dfrac{y}{d_J} \right) S_Y(y) dy - \dfrac{1}{2} \, \frac{\bet^2 \tet}{d_J} \le (1 + \tet) \la \mu - c.
\]
After substituting for $(1 + \tet) \la \mu - c$ from \eqref{eq:dJ} and rearranging terms, we obtain the following equivalent inequality:
\[
\la \int_0^{d_J} \left( (1 + \tet) - (1 + \tet)^{y/d_J} \right) S_Y(y) dy - \tet \la \int_0^{d_J} \left( 1 - \dfrac{y}{d_J} \right) S_Y(y) dy \ge - \, \dfrac{1}{2} \, \frac{\bet^2}{d_J} \big( \tet - \ln(1+\tet) \big).
\]
For $\tet \ge 0$, we know that $\tet \ge \ln(1 + \tet)$; thus, it is enough to show the following stronger inequality:
\[
\la \int_0^{d_J} \left( (1 + \tet) - (1 + \tet)^{y/d_J} \right) S_Y(y) dy - \tet \la \int_0^{d_J} \left( 1 - \dfrac{y}{d_J} \right) S_Y(y) dy \ge 0,
\]
or equivalently,
\begin{equation}\label{ineq:2}
\int_0^{d_J} \left( 1 + \dfrac{\tet y}{d_J}  - (1 + \tet)^{y/d_J} \right) S_Y(y) dy \ge 0.
\end{equation}

To show inequality \eqref{ineq:2}, we will show that the integrand is non-negative for $y \in [0, d_J]$.  To that end, define $\tilde f$ on $[0, 1]$ by
\[
\tilde{f}(x) = 1 + \tet x  - (1 + \tet)^x.
\]
By differentiating, we obtain
\[
\tilde{f}'(x) = \tet - (1 + \tet)^x \ln(1 + \tet),
\]
and
\[
\tilde{f}''(x) = - (1 + \tet)^x \big(\ln(1 + \tet)\big)^2 < 0.
\]
Also, $\tilde f(0) = 0 = \tilde f(1)$, and $\tilde f'(0) \ge 0$; thus, $\tilde f(x) \ge 0$ for all $0 \le x \le 1$, and inequality \eqref{ineq:2} follows.  We have shown that $d_J \le d_D$, which implies that $\hRj \le R_D$.
\end{proof}

When $\tet = 0$, the mean-variance premium principle reduces to the variance premium principle, but the optimal reinsurance strategy is \textit{not} proportional insurance, by contrast with Corollary \ref{cor:tet0D}.

\begin{corollary}
If $\tet = 0$, then the optimal retention strategy to maximize the adjustment coefficient is determined by
\begin{equation}\label{eq:hRj_tet0}
\hRj(y) = R_c(\rhoj, y),
\end{equation}
for all $y \ge 0$, in which $R_c(\rhoj, y) \in [0, y]$ and the maximum adjustment coefficient $\rhoj > 0$ jointly solve
\begin{equation}\label{eq:Rc_tet0}
1 + \eta y - \eta R(y) - e^{rR(y)} = 0,
\end{equation}
and
\begin{equation}\label{eq:rhoj_tet0}
c - \la \mu = r \left\{ \la \int_0^\infty \dfrac{e^{r R(y)} - 1}{r} \, S_Y(y) dy + \dfrac{\bet^2}{2} \right\}.
\end{equation}
Moreover, there exists $y_0 > 0$ such that $R_c(\rhoj, y) > R_D(y)$ if and only if $0 < y < y_0$.
\end{corollary}

\begin{proof}
Showing \eqref{eq:hRj_tet0}, \eqref{eq:Rc_tet0}, and \eqref{eq:rhoj_tet0} is an easy application of Theorem \ref{thm:rhoj}, so we omit those details.  From \eqref{eq:Rc_y}, we know that $R_c$ is increasing with respect to $y$ and
\[
\dfrac{d R_c(\rhoj, y)}{dy} \bigg|_{y = 0} = \dfrac{\eta}{\rhoj + \eta} \, .
\]
From \eqref{eq:RstarD_tet0}, we know that
\[
\dfrac{d R_D(y)}{dy} \equiv \dfrac{\eta}{\rhod + \eta} \, ,
\]
for all $y \ge 0$.  In Theorem \ref{thm:LundbergD} below, we show $\rhoj < \rhod$; thus, $R_c(\rhoj, y) > R_D(y)$ in a neighborhood of $y = 0$.  Furthermore, Corollary \ref{cor:prop_hRj} shows that $R_c(\rhoj, y)$ is strictly concave with asymptote $g(y) \big|_{\tet = 0} = \frac{1}{\rhoj} \, \ln(1 + \eta y)$; by contrast, $R_D(y)$ is linear when $\tet = 0$.  It follows that there exists $y_0 > 0$ such that $R_c(\rhoj, y) > R_D(y)$ if and only if $0 < y < y_0$.
\end{proof}

We end this section with a theorem that relates $e^{- \rhoj x}$ to the minimum probability of ruin under the diffusion approximation. We use the notation from Theorems \ref{thm:psiD} and \ref{thm:rhoj}.

\begin{theorem}\label{thm:LundbergD}
\[
\psi_D(x) = e^{-(\alp^* - \eta)x} < e^{-\rhoj x},
\]
for all $x > 0$.
\end{theorem}

\begin{proof}
To prove this theorem, we first show that, for an arbitrary retention function $R$ for which $\rhoj(R)$ exists, the following inequality holds:
\begin{equation}\label{ineq:1}
\rhoj(R) < \rhod(R).
\end{equation}
From \eqref{eq:rhoDR}, $\rhod(R)$ satisfies
\begin{equation}\label{eq:rdi}
\left[ - \kap + \la \left(\tet \E R + \eta \E(YR) - \frac{\eta}{2} \, \E \big(R^2 \big) \right) \right] - \frac{1}{2} \Big[ \bet^2 + \la \E \big(R^2 \big) \Big] r = 0.
\end{equation}
The left side of \eqref{eq:rdi} decreases with $r$; thus, if we show that the left side is positive when $r = \rhoj(R)$, then we will have proven inequality \eqref{ineq:1}.

The left side of \eqref{eq:rdi} is positive when $r = \rhoj(R)$ if and only if
\[
\left[ - \kap + \la \left(\tet \E R + \eta \E(YR) - \frac{\eta}{2} \, \E \big(R^2 \big) \right) \right] \rhoj(R) - \frac{1}{2} \Big[ \bet^2 + \la \E \big(R^2 \big) \Big] \rhoj^2(R) > 0,
\]
or equivalently, using the equation that $\rhoj(R)$ solves, namely, \eqref{eq:rhoJR},
\[
M_R \big(\rhoj(R) \big) > 1 + \E (R) \, \rhoj(R) + \dfrac{1}{2} \, \E\big(R^2\big) \rhoj^2(R).
\]
This inequality is true because $e^{x} > 1 + x + \frac{1}{2} \, x^2$ for $x > 0$.  Thus, we have shown that $\rhoj(R) < \rhod(R)$ for all retention functions for which $\rhoj(R)$ exists.

It follows that
\[
\rhoj = \rhoj(\hRj) < \rhod(\hRj) \le \rhod,
\]
in which $\hRj$ is the retention function given in \eqref{eq:hRj}.  Equation \eqref{eq:rhoD} shows that $\rhod = \alp^* - \eta$; thus, $e^{-(\alp^* - \eta)x} < e^{-\rhoj x}$ for all $x > 0$.
\end{proof}

\begin{remark}
Essentially Theorem {\rm \ref{thm:LundbergD}} follows from the fact that, for any reinsurance strategy, the adjustment coefficient under the diffusion approximation is greater than the adjustment coefficient for the classical risk process.  This inequality is pleasing because, due to its jumps, we expect a classical risk process to be ``riskier''  than its corresponding diffusion approximation, and inverse ordering of the corresponding adjustment coefficients confirms this intuitive ordering of riskiness. \qed
\end{remark}

\subsection{Scaled model}\label{sec:43}

In this section, we scale our model by $n > 0$. To obtain the scaled system, multiply the Poisson rate $\la$ by $n$, divide the claim severity by $\sqrtn$, and adjust the premium rate so that net premium income remains constant. Specifically, define $\lan = n \la$, so $n$ large is equivalent to $\lan$ large. Scale the claim severity by defining $\Yn = Y/\sqrtn$. Also, define $\tetn = \tet/\sqrtn$ and $\cn = c + ( \sqrtn -1 ) \la \E Y$, which implies $\cn - \lan \E \Yn = c - \la \E Y$, independent of $n$.  The parameters $\eta$ and $\bet$ remain unchanged.  Finally, define
\begin{align*}
\kap_n &= (1 + \tetn) \lan \E \big(\Yn\big) + \dfrac{\eta}{2} \,  \lan \E \big(\Yn^2 \big) - \cn \\
&= (1 + \tet) \la \E Y + \dfrac{\eta}{2} \,  \la \E \big(Y^2 \big) - c = \kap,
\end{align*}
also independent of $n$.

Let $\psin$ denote the minimum probability of ruin in the scaled system.  The main result of this section is that $\psi_D$ and $\psin$ are approximately equal, to order $\mO \big(n^{-1/2} \big)$.  In proving this, we also show that the maximum adjustment coefficient $\rhod$ for the diffusion approximation and the maximum adjustment coefficient for the scaled classical risk model perturbed by a diffusion $\rhojn$ are approximately equal, to order $\mO\big(n^{-1/2} \big)$, an interesting result in itself.

Consider the $R$-part of the drift of $\hX$ in \eqref{eq:Xhat}:
\begin{align}
&\lan \left( \tetn \E \Rn + \eta \E \big( \Yn \Rn \big) - \dfrac{\eta}{2} \, \E \big(\Rn^2 \big) \right) \notag \\
&= n \la \left( \dfrac{\tet}{\sqrtn} \, \E \Rn + \eta \E \bigg( \dfrac{Y}{\sqrtn} \, \Rn \bigg) - \dfrac{\eta}{2} \, \E \big( \Rn^2 \big) \right) \notag \\
&= \la \left(\tet \E \big(\sqrtn \Rn \big) + \eta \E \big( Y \, \sqrtn \Rn \big) - \dfrac{\eta}{2} \, \E \Big( \big( \sqrtn \Rn \big)^2 \Big) \right).  \label{eq:Xhat_drift}
\end{align}
Thus, we are motivated to define the retention function $\whRn$ by
\begin{equation}\label{eq:whRn}
\whRn (y) = \sqrtn \, \Rn \bigg( \dfrac{y}{\sqrtn} \bigg).
\end{equation}
Note that, if $\whRn$ is independent of $n$, then the drift of $\hX$ remains unchanged by the scaling; also, one can show that $R_D$ in \eqref{eq:RstarD} satisfies
\[
R_D(y) = \sqrtn \, R^{(n)}_D \bigg( \dfrac{y}{\sqrtn} \bigg),
\]
in which $R^{(n)}_D$ is the optimal retention function for the scaled diffusion approximation, specifically, with $\tet = \tetn$ in \eqref{eq:RstarD2}.

Let $\rhojn \big( \Rn \big)$ and $\rhodn \big( \Rn \big)$ denote the adjustment coefficients for the scaled classical risk model perturbed by a diffusion and for its diffusion approximation, respectively, under the retention function $\Rn$.  Because of the calculation in \eqref{eq:Xhat_drift} and the resulting assignment in \eqref{eq:whRn}, we deduce that
\begin{equation}\label{eq:rhodn}
\rhodn \big( \Rn \big) = \rhod \big( \whRn \big).
\end{equation}
We present the following theorem relating $\rhojn = \sup \limits_{\Rn} \rhojn \big( \Rn \big)$ and $\rhod = \sup \limits_R \rhod(R) = \sup \limits_{\whRn} \rhod \big( \whRn \big)$.

\begin{theorem}\label{thm:rhojn_lim}
Choose $C$ so that
\begin{equation}\label{eq:C}
C > \dfrac{1}{3} \, \dfrac{\la \E \big(R^3_D \big)}{\bet^2 + \la \E \big(R^2_D \big)} \, \rhod^2.
\end{equation}
Then, there exists $N > 0$ such that
\begin{equation}\label{eq:rhojn_rhod}
\rhod - \dfrac{C}{\sqrtn} < \rhojn < \rhod,
\end{equation}
for all $n > N$, from which it follows that
\begin{equation}\label{eq:rhojn_lim}
\lim \limits_{n \to \infty} \rhojn = \rhod,
\end{equation}
with rate of convergence of order $\mO \big( n^{-1/2} \big)$.
\end{theorem}

\begin{proof}
First, note that the limit in \eqref{eq:rhojn_lim} follows directly from the bounds in \eqref{eq:rhojn_rhod}.  Next, note that the second inequality in \eqref{eq:rhojn_rhod} follows from the proof of Theorem \ref{thm:LundbergD}.  Thus, to prove this theorem it is enough to prove the first inequality in \eqref{eq:rhojn_rhod}.

First, fix a retention function $\Rn$.  Because we assume $M_Y(r)$ exists in a neighborhood of $0$, $M_{\Rn}(r)$ also exists in a neighborhood of $0$. From \eqref{eq:rhoJR}, we know that $\rhojn \big(\Rn \big)$ solves
\[
\left[- \kapn + \lan \left( (1 + \tetn) \E \Rn + \eta \E \big(\Yn \Rn \big) - \frac{\eta}{2} \, \E\big( \Rn^2 \big) \right) \right] r - \frac{1}{2} \, \bet^2 r^2 - \lan \big( M_{\Rn}(r) - 1 \big) = 0,
\]
or equivalently, from \eqref{eq:whRn} and the scaling of the model,
\begin{equation}\label{eq:rhoJRn}
\left[- \kap + \la \left( \tet \E \whRn + \eta \E \big(Y \whRn \big) - \frac{\eta}{2} \, \E\big(\whRn^2\big) \right) \right] - \frac{1}{2} \, \bet^2 r - n \la \, \dfrac{ \E e^{\frac{r}{\sqrtn} \, \whRn} - 1 - \frac{r}{\sqrtn} \, \E \whRn}{r} = 0.
\end{equation}
The left side of \eqref{eq:rhoJRn} decreases with respect to $r$; thus, to show that $\rhod \big(\whRn \big) - C \big(\whRn \big)/\sqrtn < \rhojn \big( \Rn \big)$ for some $C\big(\whRn \big) > 0$, it is enough to show that the left side of \eqref{eq:rhoJRn} is positive at $r = \rhod \big( \whRn \big) - C\big( \whRn \big)/\sqrtn$, that is,
\begin{align*}
&\left[- \kap + \la \left( \tet \E \whRn + \eta \E \big(Y \whRn \big) - \frac{\eta}{2} \, \E\big(\whRn^2\big) \right) \right] - \frac{1}{2} \, \bet^2 \left( \rhod \big( \whRn \big) - \dfrac{C \big( \whRn \big)}{\sqrtn} \right)  \\
&\quad  - n \la \, \dfrac{\E \exp \left( \frac{\rhod \big( \whRn \big) - \frac{C \big( \whRn \big)}{\sqrtn}}{\sqrtn} \, \whRn \right) - 1 - \frac{\rhod \big( \whRn \big) - \frac{C \big( \whRn \big)}{\sqrtn}}{\sqrtn} \, \E \whRn}{\rhod \big( \whRn \big) - \frac{C \big( \whRn \big)}{\sqrtn}} > 0.
\end{align*}
By using the equation for $\rhod(R)$, namely, \eqref{eq:rhoDR}, and by expanding the exponential, we see that this inequality is equivalent to
\begin{align*}
\left( \Big( \bet^2 + \la \E\big(\whRn^2\big) \Big) C \big( \whRn \big) - \dfrac{\la}{3} \, \E\big(\whRn^3 \big) \big(\rhod \big( \whRn \big) \big)^2 \right) + \mO \big( n^{-1/2} \big) > 0,
\end{align*}
in which the $\mO \big( n^{-1/2} \big)$ term is finite for $n$ large enough because of the assumption that $M_{\Rn}(r)$ (and, hence, $M_{\whRn}(r)$) exists for $r$ in a neighborhood of $0$.  So, if we choose $C\big( \whRn \big)$ such that
\[
C\big( \whRn \big) > \dfrac{1}{3} \, \dfrac{\la \E \big(\whRn^3 \big)}{\bet^2 + \la \E \big(\whRn^2 \big)} \,  \big(\rhod\big( \whRn \big) \big)^2,
\]
then there exists $N\big( \whRn \big) > 0$ such that $n > N\big( \whRn \big)$ implies
\[
\rhod\big( \whRn \big) - \dfrac{C\big( \whRn \big)}{\sqrtn} < \rhojn \big( \Rn \big).
\]
Because $\rhod = \rhod \big( R_D \big)$, that is, the arg sup over $\whRn$ is independent of $n$, it follows that
\[
\rhod - \dfrac{C\big( R_D \big)}{\sqrtn} < \rhojn \Bigg( \dfrac{1}{\sqrtn} \, R_D \Bigg) \le \sup_{\Rn} \rhojn \big( \Rn \big) = \rhojn.
\]
Thus, we have proved the first inequality in \eqref{eq:rhojn_rhod}, and we have completed the proof of this theorem.
\end{proof}

We readily obtain the following corollary of this theorem, in which we modify $\psi_D(x) = e^{-\rhod \, x}$ by a function of order $\mO \big( n^{-1/2} \big)$ to obtain an upper bound of $\psin$.

\begin{corollary}\label{cor:psin_upper}
Let $C$ and $N$ be as in the statement of Theorem {\rm \ref{thm:rhojn_lim}}. Then, for all $n > N$,
\begin{equation}\label{eq:psin_upper}
\psin(x) < e^{- \left( \rhod - \frac{C}{\sqrtn} \right)x},
\end{equation}
for all $x > 0$.
\end{corollary}

\begin{proof}
For $n = 1$, let $\psi_{\hRj}$ denote the probability of ruin when the insurer follows the strategy that maximizes the adjustment coefficient; thus, $\psi(x) \le \psi_{\hRj}(x) < e^{- \rhoj \, x}$, in which the second inequality follows from the Lundberg bound.  It follows that, for $n \in \N$, $\psin(x) < e^{-\rhojn \, x}$ for all $x > 0$.  Inequality \eqref{eq:rhojn_rhod}, then, implies \eqref{eq:psin_upper}.
\end{proof}

In the next theorem, we look at the limit of the corresponding optimal retention strategies.  Let $\hRjn$ denote the optimal retention function to maximize $\rhojn \big( \Rn \big)$.

\begin{theorem}\label{thm:hRjn_RD}
\begin{equation}\label{eq:hRjn_RD}
\lim \limits_{n \to \infty} \sqrtn \, \hRjn \bigg( \dfrac{y}{\sqrtn} \bigg) = R_D(y),
\end{equation}
for all $y \ge 0$.
\end{theorem}

\begin{proof}
Note that
\[
\dfrac{y}{\sqrtn} \ge \dfrac{1}{\rhojn} \, \ln \big(1 + \tetn \big)
\]
if and only if
\begin{equation}\label{eq:y_lim}
y \ge \dfrac{\sqrtn}{\rhojn} \, \ln \big(1 + \tetn \big),
\end{equation}
and the right side goes to $\frac{\tet}{\rhod}$ as $n \to \infty$.  From \eqref{eq:Rc}, for $y \ge \frac{\sqrtn}{\rhojn} \, \ln \big(1 + \tetn \big)$,
\[
\left(1 + \dfrac{\tet}{\sqrtn} \right) + \eta \, \dfrac{y}{\sqrtn} - \eta R_c \bigg( \rhojn, \, \dfrac{y}{\sqrtn} \bigg) - \exp \left( \rhojn R_c\bigg( \rhojn, \, \dfrac{y}{\sqrtn} \bigg)\right) = 0,
\]
which, after expanding the exponential, implies
\[
- \big(\rhojn + \eta \big) \big( \sqrtn \, R_c \big) + \tet + \eta y + \mO \big(n^{-1/2} \big) = 0.
\]
By taking a limit as $n \to \infty$, we obtain
\begin{equation}\label{eq:Rc_lim}
\lim_{n \to \infty} \sqrtn \, R_c \bigg( \rhojn, \, \dfrac{y}{\sqrtn} \bigg) = \dfrac{\tet + \eta y}{\rhod + \eta} \, .
\end{equation}
Equation \eqref{eq:hRjn_RD}, then, follows from the expression for $\hRjn$ in \eqref{eq:hRj}, from the expression for $R_D$ in \eqref{eq:RstarD2}, from $\rhod = \alpha^* - \eta$, from \eqref{eq:y_lim}, and from \eqref{eq:Rc_lim}.
\end{proof}

We end this section with an important result, namely, that as $n \to \infty$, $\psin$ converges to $\psi_D$ uniformly with respect to $x$ and with rate of convergence of order $\mO \big(n^{-1/2} \big)$.  In Corollary \ref{cor:psin_upper}, we modified $\psi_D$ by a function of order $\mO \big( n^{-1/2} \big)$ to obtain an {\it upper} bound of $\psin$.  In the next proposition, we modify $\psi_D$ by a function of order $\mO \big( n^{-1/2} \big)$ to obtain a {\it lower} bound of $\psin$.  After that, we show how convergence follows from these two bounds.

\begin{proposition}\label{prop:psin_lower}
Formally, define the random variable $\Zd = (Y - d) \big| (Y > d)$ for $d \ge 0$, and suppose $m$ exists such that $M_Y \big(\rhod/\sqrt{m} \, \big) < \infty$, with
\begin{equation}\label{eq:Zd}
\sup \limits_{d \ge 0} \E \Big( \Zd^2 \, e^{\frac{\rhod}{\sqrt{m}} \Zd} \Big) < \infty.
\end{equation}
Choose $\eps > 0$, and define $\del$ by
\begin{equation}\label{eq:del}
\del = \sup \limits_{d \ge 0} \big( \rhod \E \Zd + \eps \big),
\end{equation}
and choose $N > \max \big( \del^2, m \big)$ such that
\begin{equation}\label{eq:N}
\sup \limits_{d \ge 0} \dfrac{\rhod^2}{\sqrt{N}} \int_0^1 (1 - \ome) \E\Big( \Zd^2 \, e^{\frac{\rhod \ome}{\sqrt{N}} \Zd} \Big) d \ome \le \eps.
\end{equation}
Then, for all $n > N$,
\begin{equation}\label{eq:psio_scale}
\left( 1 - \dfrac{\del}{\sqrtn} \right) \psi_D(x) \le \psin(x),
\end{equation}
for all $x > 0$.
\end{proposition}

\begin{proof}
Without loss of generality, assume $n > \del^2$, and define the function $\elln$ on $\R$ by
\begin{equation}\label{eq:elln}
\elln(x) =
\begin{cases}
1, &\quad x \le 0, \\
\left( 1 - \dfrac{\del}{\sqrtn} \right) e^{-\rhod x},   &\quad x > 0.
\end{cases}
\end{equation}
It follows that $\elln(x) = 1 = \psin(x)$ for all $x \le 0$, and $\lim \limits_{x \to \infty} \elln(x) = 0 = \lim \limits_{x \to \infty} \psin(x)$.   Thus, by Theorem \ref{thm:comp} (using $\elln$ in the role of an upper semi-continuous subsolution), to show that $\elln(x) \le \psin(x)$ for $x > 0$, it enough to show, for all $x > 0$,
\begin{equation}\label{eq:Fn_elln_neg}
\Fn \big(x, \elln(x), \elln'(x), \elln''(x), \elln(\cdot) \big) \le 0,
\end{equation}
in which $\Fn$ denotes the operator in \eqref{eq:F} corresponding to the scaled model.  Recall that \hfill \break $\sqrtn \, \Rn(y/\sqrtn \, ) = \whRn(y)$. For $x > 0$,
\begin{align}\label{eq:Fn_elln}
&\Fn \big( x, \elln(x), \elln'(x), \elln''(x), \elln(\cdot) \big) = \kap \elln'(x) - \dfrac{1}{2} \, \bet^2 \elln''(x) \notag \\
&\quad - n\la \inf \limits_{\Rn} \left\{ \left( \left( 1 + \dfrac{\tet}{\sqrtn} \right) \E \Rn + \eta \E \bigg( \dfrac{Y}{\sqrtn} \, \Rn \bigg) - \dfrac{\eta}{2} \, \E \big( \Rn^2 \big) \right) \elln'(x) + \E \elln \big( x - \Rn \big) - \elln(x) \right\} \notag \\
&= \kap \elln'(x) - \dfrac{1}{2} \, \bet^2 \elln''(x) \notag \\
&\quad - \la \inf \limits_{\whRn} \left\{ \left( \big( \sqrtn + \tet \big) \E \whRn + \eta \E \big( Y \whRn \big) - \dfrac{\eta}{2} \, \E \big( \whRn^2 \big) \right) \elln'(x) + n \E \elln \bigg( x - \dfrac{\whRn}{\sqrtn} \bigg) - n \elln(x) \right\} \notag \\
&= - \left( 1 - \dfrac{\del}{\sqrtn} \right) \kap \rhod e^{-\rhod \, x} - \dfrac{1}{2} \left( 1 - \dfrac{\del}{\sqrtn} \right) \bet^2 \rhod^2 e^{-\rhod \, x} \notag \\
&\quad - \la \left( 1 - \dfrac{\del}{\sqrtn} \right) e^{-\rhod \, x} \inf \limits_{\whRn} \Bigg\{ - \left( \big( \sqrtn + \tet \big) \E \whRn + \eta \E \big( Y \whRn \big) - \dfrac{\eta}{2} \, \E \big( \whRn^2 \big) \right) \rhod \notag  \\
&\qquad \qquad \qquad \qquad \qquad \qquad \left. + n \int_0^\infty \left[ e^{\frac{\rhod}{\sqrtn} \, \whRn(y)} {\bf 1}_{\{\whRn(y) \le \sqrtn x \}} + \dfrac{e^{\rhod x}}{1 - \frac{\del}{\sqrtn}} \, {\bf 1}_{\{ \whRn(y) > \sqrtn x \}} - 1 \right] dF_Y(y) \right\} \notag \\
&\propto - \kap \rhod - \dfrac{1}{2} \, \bet^2 \rhod^2 - \la \inf \limits_{\whRn} \Bigg\{ - \left( \big( \sqrtn + \tet \big) \E \whRn + \eta \E \big( Y\whRn \big) - \dfrac{\eta}{2} \, \E \big( \whRn^2 \big) \right) \rhod \notag \\
&\qquad \qquad \qquad \qquad \qquad \qquad + n \int_0^\infty \left[ e^{\frac{\rhod}{\sqrtn} \, \whRn(y)} {\bf 1}_{\{\whRn(y) \le \sqrtn x \}} + \dfrac{e^{\rhod x}}{1 - \frac{\del}{\sqrtn}} \,  {\bf 1}_{\{ \whRn(y) > \sqrtn x \}} - 1 \right] dF_Y(y) \Bigg\}.
\end{align}

Recall that $\rhod$ solves
\[
- \kap \rhod - \dfrac{1}{2} \, \bet^2 \rhod^2 = \la \inf \limits_R \left\{ - \left( \tet \E R + \eta \E \big( YR \big) - \dfrac{\eta}{2} \, \E \big( R^2 \big) \right) \rhod + \dfrac{1}{2} \, \E \big(R^2 \big) \rhod^2 \right\}.
\]
Thus, inequality \eqref{eq:Fn_elln_neg} holds if and only if
\begin{align}\label{eq:Fn_elln_neg2}
&\inf \limits_{\whRn} \left\{ - \left( \tet \E \whRn + \eta \E \big( Y\whRn \big) - \dfrac{\eta}{2} \, \E \big( \whRn^2 \big) \right) \rhod + \dfrac{1}{2} \, \E \big(\whRn^2 \big) \rhod^2 \right\} \\
&\le \inf \limits_{\whRn} \Bigg\{ - \left( \tet \E \whRn + \eta \E \big( Y\whRn \big) - \dfrac{\eta}{2} \, \E \big( \whRn^2 \big) \right) \rhod \notag \\
&\quad \qquad + n \int_0^\infty \left[ e^{\frac{\rhod}{\sqrtn} \, \whRn(y)} {\bf 1}_{\{\whRn(y) \le \sqrtn x \}} + \dfrac{e^{\rhod x}}{1 - \frac{\del}{\sqrtn}} \, {\bf 1}_{\{ \whRn(y) > \sqrtn x \}} - 1 - \dfrac{\rhod \whRn(y)}{\sqrtn} \right] dF_Y(y) \Bigg\}. \notag
\end{align}
To show \eqref{eq:Fn_elln_neg2}, consider a fixed retention function $\whRn$; we wish to show that this inequality holds for $\whRn$, or equivalently,
\begin{equation}\label{eq:Fn_elln_neg3}
\dfrac{1}{2} \, \E \big(\whRn^2 \big) \rhod^2 \le n \int_0^\infty \left[ e^{\frac{\rhod}{\sqrtn} \, \whRn(y)} {\bf 1}_{\{\whRn(y) \le \sqrtn x \}} + \dfrac{e^{\rhod x}}{1 - \frac{\del}{\sqrtn}} \, {\bf 1}_{\{ \whRn(y) > \sqrtn x \}} - 1 - \dfrac{\rhod \whRn(y)}{\sqrtn} \right] dF_Y(y).
\end{equation}
The optimal retention function for the left side of \eqref{eq:Fn_elln_neg2} is $R_D$ in \eqref{eq:RstarD}, which is strictly increasing with respect to $y$ with rate of increase less than $1$ and $R_D$ increases without bound.  Moreover, via a proof similar to the proof of Theorem \ref{thm:rhoj}, one can show that the optimal retention function for the right side of  \eqref{eq:Fn_elln_neg2} satisfies those same properties.  It follows that, without loss of generality, we can assume $\whRn$ in \eqref{eq:Fn_elln_neg3} is strictly increasing with respect to $y$ with rate of increase less than $1$, and $\whRn$ increases without bound.  Thus, define $\ynx$ by
\begin{equation}\label{eq:ynx}
\ynx = \whRn^{-1} (\sqrtn \, x),
\end{equation}
and inequality \eqref{eq:Fn_elln_neg3} becomes
\[
\dfrac{1}{2} \, \E \big(\whRn^2 \big) \rhod^2 \le n  \int_0^\infty  \left( e^{\frac{\rhod \whRn(y)}{\sqrtn}} - 1 - \dfrac{\rhod \whRn(y)}{\sqrtn} \right) dF_Y(y) + n\int_{\ynx}^\infty \left( \dfrac{e^{\rhod x}}{1 - \frac{\del}{\sqrtn}} - e^{\frac{\rhod \whRn(y)}{\sqrtn}} \right) dF_Y(y).
\]
Because $e^x > 1 + x + x^2/2$ for all $x > 0$, we know that
\[
\dfrac{1}{2} \, \E \big(\whRn^2 \big) \rhod^2 \le n  \int_0^\infty  \left( e^{\frac{\rhod \whRn(y)}{\sqrtn}} - 1 - \dfrac{\rhod \whRn(y)}{\sqrtn} \right) dF_Y(y).
\]

It remains to find values of $\del$ and $N > \del^2$ for which
\[
\int_{\ynx}^\infty \left( \dfrac{e^{\rhod x}}{1 - \frac{\del}{\sqrtn}} - e^{\frac{\rhod \whRn(y)}{\sqrtn}} \right) dF_Y(y) \ge 0,
\]
for all $n > N$ and for all $x > 0$.  If $S_Y \big(\ynx \big) = 0$, then the left side is identically $0$, so suppose that $S_Y \big(\ynx \big) > 0$.  After replacing $\ynx$ by $d$ and dividing by $e^{\rhod x} S_Y(d)$, the above inequality becomes
\[
\int_d^\infty \left( \dfrac{1}{1 - \frac{\del}{\sqrtn}} - e^{ \frac{\rhod}{\sqrtn}\big(\whRn(y) - \whRn(d) \big)} \right) \dfrac{dF_Y(y)}{S_Y(d)} \ge 0,
\]
for $d \ge 0$, or equivalently,
\[
\int_d^\infty \left( e^{ \frac{\rhod}{\sqrtn}\big(\whRn(y) - \whRn(d)\big)} - 1 \right) \dfrac{dF_Y(y)}{S_Y(d)} \le \dfrac{\frac{\del}{\sqrtn}}{1 - \frac{\del}{\sqrtn}} \, .
\]
Because, without loss of generality, $0 \le \whRn(y) - \whRn(d) \le y - d$ for all $y \ge d$, it is enough to find $\del$ and $N > \del^2$ (with $N$ independent of $d$) so that the following stronger inequality holds:
\[
\int_d^\infty \left( e^{ \frac{\rhod}{\sqrtn}(y - d)} - 1 \right) \dfrac{dF_Y(y)}{S_Y(d)} \le \dfrac{\frac{\del}{\sqrtn}}{1 - \frac{\del}{\sqrtn}} \, ,
\]
for all $n > N$.  To that end, define $\Zd = (Y - d) \big| (Y > d)$; then, this stronger inequality becomes
\[
\int_0^\infty \left(  e^{ \frac{\rhod z}{\sqrtn}} - 1 \right) dF_{\Zd}(z) \le \dfrac{\frac{\del}{\sqrtn}}{1 - \frac{\del}{\sqrtn}} \, .
\]
If we find $\del$ that satisfies the following even stronger inequality, then the above sequence of inequalities holds:
\begin{equation}\label{ineq:del1}
\int_0^\infty \left(  e^{ \frac{\rhod z}{\sqrtn}} - 1 \right) dF_{\Zd}(z) \le \dfrac{\del}{\sqrtn} \, .
\end{equation}
Rewrite the integrand from the left side of inequality \eqref{ineq:del1} as follows:
\[
e^{ \frac{\rhod z}{\sqrtn}} - 1 = \dfrac{\rhod z}{\sqrtn} + \dfrac{\rhod^2 z^2}{n} \int_0^1 (1 - \ome) e^{\frac{\rhod z}{\sqrtn} \, \ome} d\ome.
\]
Thus, inequality \eqref{ineq:del1} is equivalent to
\[
\int_0^\infty \left(  \dfrac{\rhod z}{\sqrtn} + \dfrac{\rhod^2 z^2}{n} \int_0^1 (1 - \ome) e^{\frac{\rhod z}{\sqrtn} \, \ome} d \ome \right) dF_{\Zd}(z) \le \dfrac{\del}{\sqrtn} \, ,
\]
or, after multiplying both side by $\sqrtn$ and switching the order of integration,
\begin{equation}\label{ineq:del2}
\rhod \E \Zd + \dfrac{\rhod^2}{\sqrtn} \int_0^1 (1 - \ome) \, \E \Big( \Zd^2 \, e^{\frac{\rhod \ome}{\sqrt{n}} \Zd} \Big) d \ome \le \del.
\end{equation}
Note that the left side decreases with increasing $n$.  It follows that if we define $\del$ and $N$ as in \eqref{eq:del} and \eqref{eq:N}, respectively, then inequality \eqref{ineq:del2} holds for all $d \ge 0$ and all $n > N$, which implies that $F_n$ evaluated at $\elln$ is non-positive for all $x > 0$ and all $n > N$.  The conclusion in \eqref{eq:psio_scale}, then, follows from Theorem \ref{thm:comp}.
\end{proof}

In the following theorem, we combine the results of Corollary \ref{cor:psin_upper} and Proposition \ref{prop:psin_lower}.

\begin{theorem}\label{thm:psin_lim}
If \eqref{eq:Zd} holds,
then there exist $C' > 0$ and $N > 0$ such that, for all $n > N$ and $x \ge 0$,
\begin{align}\label{ASAF1}
\big| \psin(x) - \psi_D(x) \big| \le \dfrac{C'}{\sqrtn} \, .
\end{align}
Recall from \eqref{eq:psiD2} that $\psi_D(x) = e^{-\rhod x}$, with $\rhod$ solving \eqref{eq:rhoD}.
\end{theorem}

\begin{proof}
From Corollary \ref{cor:psin_upper} and Proposition \ref{prop:psin_lower} it follows that
$$
\left(1 - \dfrac{\del}{\sqrtn}\right) e^{-\rhod x} < \psin(x) < e^{-\left(\rhod - \frac{C}{\sqrtn} \right)x}.
$$
Subtracting $e^{-\rhod x}$ from each side yields,
\begin{align}\notag
- \, \dfrac{\del}{\sqrtn} \, e^{-\rhod x} < \psin(x) - e^{-\rhod x} < e^{-\left(\rhod - \frac{C}{\sqrtn} \right)x} - e^{-\rhod x}.
\end{align}
Clearly, the left side is bounded below by $-\del/\sqrtn$. Basic calculus implies that, for every $n > \big(C/\rhod \big)^2$, the right side is bounded above by
$$
\left(1 - \frac{C}{\rhod \sqrtn} \right)^{\frac{\rhod \sqrtn}{C}} \left(\frac{\frac{C}{\sqrtn}}{{\rhod - \frac{C}{\sqrtn}}} \right).
$$
The first factor converges to $e^{-1}$, and the second factor is of order $\mO\big(n^{-1/2} \big)$. By combining this upper bound with the lower bound, we deduce inequality \eqref{ASAF1}.
\end{proof}

\begin{remark}
Theorem {\rm \ref{thm:psin_lim}} partially justifies using the diffusion approximation for the classical risk process when $\la$ is large.  However, it does {\rm not} prove that the optimal retention strategy for $\psin$ converges to $R_D$, as we showed in Theorem {\rm \ref{thm:hRjn_RD}} when maximizing the adjustment coefficient.  If one were to show the limit in \eqref{eq:hRjn_RD} also holds for the optimal retention strategy for $\psin$, then that would fully justify using the diffusion approximation for the classical risk process when $\la$ is large.  \qed
\end{remark}

We end this section with a corollary that estimates the degree to which optimal reinsurance decreases the probability of ruin.  From Theorem 4.1 in Cohen and Young \cite{CY2019}, adapted to our model with the diffusion perturbation, we know that, if the analog of \eqref{eq:Zd} holds in the uncontrolled case, then there exists $D > 0$ and $M > 0$ such that, for all $n > M$ and $x \ge 0$,
\begin{align}\label{ASAF1_2}
\left| \psin^0(x) - e^{- \rhod(Y) x} \right| \le \dfrac{D}{\sqrtn} \, ,
\end{align}
in which $\psin^0$ denotes the uncontrolled probability of ruin in the scaled model, and
\begin{equation}\label{eq:gam}
\rhod(Y) = 2 \, \dfrac{c - \la \E Y}{\la \E \big(Y^2 \big) + \bet^2} \, .
\end{equation}
Equations \eqref{eq:adjcoeff} and \eqref{eq:rhoDR} show that $e^{- \rhod(Y) x}$ is the uncontrolled probability of ruin for the diffusion approximation model.  Note that $\psin(x) \le \psin^0(x)$ and $e^{-\rhod x} \le e^{-\rhod(Y) x}$ because the minimum probability of ruin is always bounded above by the corresponding uncontrolled probability of ruin.

\begin{corollary}
If \eqref{eq:Zd} holds, then there exists $J > 0$ and $Q > 0$ such that, for all $n > Q$ and $x \ge 0$,
\begin{align}\label{ASAF1_3}
0 \le \psin^0(x) - \psin(x) \le \dfrac{J}{\sqrtn} + \left( e^{-\rhod(Y) x} - e^{-\rhod x} \right).
\end{align}
\end{corollary}

\begin{proof}
From Theorem \ref{thm:psin_lim} and from Theorem 4.1 in Cohen and Young \cite{CY2019}, there exist $C' > 0$, $D > 0$, and $Q = \max(N, M)$ such that
\begin{align*}
0 \le \psin^0(x) - \psin(x) &= \left(\psin^0(x) - e^{-\rhod(Y) x} \right) + \left( e^{-\rhod(Y) x} - e^{-\rhod x} \right) + \left( e^{-\rhod x} - \psin(x) \right) \\
&\le \dfrac{C' + D}{\sqrtn} + \left( e^{-\rhod(Y) x} - e^{-\rhod x} \right).
\end{align*}
Thus, \eqref{ASAF1_3} follows if we set $J = C' + D$.
\end{proof}

\appendix

\section{Proof that $\psi$ is the unique viscosity solution of \eqref{eq:hjbx}}\label{app:A}

\setcounter{equation}{0}
\renewcommand{\theequation}
{A.\arabic{equation}}

For the classical risk model perturbed by a diffusion, we cannot find an explicit expression for the minimum probability of ruin.  Instead, we apply stochastic Perron's method, created by Bayraktar and S\^irbu \cite{BS2012, BS2013}, to prove that the value function $\psi$ is the unique continuous viscosity solution of its HJB equation with appropriate boundary conditions.  Define the operator $F$ via its action on appropriately differentiable functions $u$ and $v$ as follows:  For $x > 0$,
\begin{align}
\label{eq:F}
& F \big(x, u(x), v_x(x), v_{xx}(x), u(\cdot) \big) = \kap v_x(x) - \dfrac{1}{2} \, \bet^2 v_{xx}(x)  \notag \\
& \qquad - \la \inf_{R} \left[ \left( (1 + \tet) \E R + \eta \E(YR) - \dfrac{\eta}{2} \, \E \big(R^2 \big) \right)
v_x(x) + {\E u(x - R) - u(x)} \right],
\end{align}
in which we take the infimum over Borel-measurable retention functions $R$ such that $0 \le R(y) \le y$ for all $y \in \R^+$.  Then, the HJB equation for our problem is
\begin{align}
\label{HJB_F}
F\big(x, v(x), v_x(x), v_{xx}(x), v(\cdot) \big) = 0,  \quad \text{for all } x > 0,
\end{align}
with boundary conditions
\begin{align}
\label{boundary_condition}
v(x) = 1 \text{ for all } x \le 0, \quad \text{ and } \quad \lim \limits_{x \to \infty} v(x) = 0.
\end{align}

Next, we define viscosity sub- and supersolutions for our problem, but, first, we recall the definitions of subjets and superjets.

\begin{defin}\label{def:sup/subjet1}
Let $v: \R \to [0,1]$ be a u.s.c.\ function, and let $u: \R \to [0,1]$ be an l.s.c.\ function.
\begin{enumerate}
\item[$1.$]  $(p,X) \in \R^2 $ is a {\rm second-order superjet} of $v$ at $x \in \R $ if
\[
v(x+z)\le v(x) + p z + \frac{1}{2} \, Xz^2 + o\big(z^2 \big), \quad \text{as } z \to 0.
\]
\item[$2.$]   $(p,X) \in \R^2 $ is a {\rm second-order subjet} of $u$ at $x \in \R $ if
\[
u(x+z)\ge u(x) + p z + \frac{1}{2} \, Xz^2 + o \big(z^2 \big), \quad \text{as } z \to 0.
\]
\item[$3.$]  $(p,X) \in \R^2 $ is a {\rm limiting superjet} of $v$ at $x$ if there exists a seqence $\{(x_n, p_n, X_n)\} \to (x, p, X)$ such that $(p_n, X_n)$ is a superjet of $v$ at $x_n$ and $v(x_n) \to v(x)$.
\item[$4.$]  $(p,X) \in \R^2 $ is a {\rm limiting subjet} of $u$ at $x$ if
there exists a sequence $\{(x_n, p_n, X_n)\} \to (x, p, X)$ such that $(p_n, X_n)$ is a subjet of $u$ at $x_n$ and $u(x_n) \to u(x)$.
\end{enumerate}
$J^{2,+} v(x)$, $\bar{J}^{2,+} v(x)$ denote the sets of second-order superjets and limiting superjets of $v$ at $x$, respectively, and $J^{2,-} u(x)$, $\bar{J}^{2,-} u(x)$ denote the sets of second-order subjets and limiting subjets of $u$ at $x$, respectively.   \qed
\end{defin}

\begin{defin}\label{def:sup/subjet}
We say a u.s.c.\ function $\underline{u}: \R \to [0, 1]$ is a {\rm viscosity subsolution} of \eqref{HJB_F} and \eqref{boundary_condition} if \eqref{boundary_condition} holds and if, for any $x_0 > 0$ and $(p,X)\in J^{2,+} \underline{u}(x_0)$, we have
\begin{equation}\label{eq:subjet_ineq}
F\big(x_0, \underline{u}(x_0), p, X, \underline{u}(\cdot) \big) \le 0.
\end{equation}
Similarly, we say an l.s.c.\ function $\bar{u}: \R \to [0, 1]$ is a {\rm viscosity supersolution} of \eqref{HJB_F} and \eqref{boundary_condition} if \eqref{boundary_condition} holds and if, for any $x_0 > 0$ and $(p,X)\in J^{2,-} \bar{u}(x_0)$, we have
\begin{equation}\label{eq:superjet_ineq}
F\big(x_0, \bar{u}(x_0), p, X, \bar{u}(\cdot) \big) \ge 0.
\end{equation}
Finally, a function $u$ is called a $($continuous$)$ {\rm viscosity solution} of \eqref{HJB_F} and \eqref{boundary_condition} if it is both a viscosity subsolution and a viscosity supersolution of \eqref{HJB_F} and \eqref{boundary_condition}.  \qed
\end{defin}

\begin{remark}
Note that, because the operator $F$ in our model is continuous, we may replace $J^{2,+} \underline{u}(x_0)$ and $J^{2,-} \bar{u}(x_0)$ with their corresponding closures $\bar{J}^{2,+} \underline{u}(x_0)$ and $\bar{J}^{2,-} \bar{u}(x_0)$ in Definition {\rm \ref{def:sup/subjet}}. \qed
\end{remark}

\begin{remark}
We could allow non-strict equality in the boundary conditions, that is, $\underline{u}(x) \le 1$ for all $x \le 0$, $\lim \limits_{x \to \infty} \underline{u}(x) \le 0$, $\bar{u}(x) \ge 1$ for all $x \le 0$, and $\lim \limits_{x \to \infty} \bar{u}(x) \ge 0$, but it is useful in what follows to require strict equality.  \qed
\end{remark}

We use stochastic Perron's method, introduced by Bayraktar and S\^irbu \cite{BS2012, BS2013}, to construct a solution of the HJB equation and, then, use a comparison theorem to verify that this solution equals the value function.  The main arguments are as follows:  First, we bound the value function from below and above by {\it stochastic} sub- and supersolutions (as defined later in this appendix):
\begin{align}
\label{eq:u_vs_v}
u \le \psi \le v.
\end{align}
Let ${\Psi^{-}}$ and ${\Psi^{+}}$ denote the sets of stochastic sub- and supersolutions, respectively.  Define $u_{-}$ and $v_{+}$ on $\R^+$ by
\[
u_{-}(x) = \sup_{u \in \Psi^{-}} u(x), \qquad \qquad v_{+}(x) = \inf_{v \in \Psi^{+}} v(x).
\]
From \eqref{eq:u_vs_v}, we deduce
\begin{align*}
u_{-} \le \psi \le v_{+}.
\end{align*}
Second, we prove that $u_{-}$ is a viscosity supersolution and $v_{+}$ is a viscosity subsolution of \eqref{HJB_F} and \eqref{boundary_condition}.  Third, a comparison result for viscosity sub- and supersolutions implies the reverse inequality, namely,
\[
u_{-} \ge v_{+}.
\]
Thus, we conclude that $\psi (= u_{-} = v_{+})$ is the unique (continuous) viscosity solution of the HJB equation satisfying the boundary conditions in \eqref{boundary_condition}.

\subsection{Stochastic supersolution}

To apply stochastic Perron's method, we first redefine the stochastic control problem using a stronger formulation.  To that end, let $0 \le \tau \le \omega \le \tau_0$ be stopping times.  Recall that $\tau_0$ is the time of ruin, defined in \eqref{eq:tau}.
Let $\fR_{\tau, \omega}$ denote the collection of predictable processes $\mR: (\tau, \omega] \to \R^+$, by which we mean that, for a fixed value of $y \ge 0$, the mapping $(t, \varpi) \mapsto R_t(\varpi, y) \times {\bf 1}_{\{ \tau < t \le \omega \}}$ is predictable with respect to the filtration $\mathbb{F}$, and $0 \le R_t(\varpi, y) \le y$ for all $t$ in the stochastic interval $(\tau, \omega]$ and $\varpi \in \Omega$.

\begin{defin}
A pair $(\tau, \zeta)$ is called a {\rm random initial condition} if $\tau$ is an $\mathbb{F}$-stopping time taking values in $[0, \tau_0]$ and $\zeta$ is an $\mF_{\tau}$-measurable random variable.  Then, for $\mR \in \fR_{\tau, \tau_0}$, the insurer's surplus process ${X}^{\tau, \zeta, \mR}$ is given by, for $t \in [\tau, \tau_0)$,
\begin{align}
\label{eq:X_rand}
X^{\tau, \zeta, \mR}_t
&= \zeta + \int^t_{\tau} \left( - \kap + \la \left( (1 + \tet) \E R_s + \eta \E \big( Y R_s \big) - \dfrac{\eta}{2} \, \E \big( R^2_s \big) \right) \right)d s - \int^t_{\tau} R_s dN_s + \bet (W_t - W_\tau).
\end{align}
\end{defin}

For convenience in what follows, we introduce a so-called {\it coffin state} $\Delta$, which represents the state when ruin occurs.  We set $\Delta + x = \Delta$ for all $x \in \R$ and $X_t = \Delta$ for all $ t \in [\tau_0, \infty)$.  For any function $u$ defined on $\R$, we extend it to $\R \cup \{ \Delta \}$ by setting $u(\Delta) = 1$.  Next, we define a stochastic supersolution.

\begin{defin} \label{def:stoch_super}
A u.s.c.\ function $v: \R \to \left[0, 1 \right]$ is called a {\rm stochastic supersolution} if it satisfies the following properties:
\begin{itemize}
\item[$(1)$]
For any random initial condition $(\tau, \zeta)$, there exists a retention strategy $\mR \in \fR_{\tau, \tau_0}$ such that, for any $\mathbb{F}$-stopping time $\omega \in [\tau, \tau_0]$,
\[
v(\zeta) \ge \E \left[ v \big({X}^{\tau, \zeta, \mR}_\omega \big) \Big | \, \mF_{\tau}\right] ~~ a.s,
\]
in which $v$ is understood to be its extension to $\R \cup \{\Delta \}$.  We say that $\mR$ is {\rm associated with} $v$ for the initial condition $(\tau, \zeta)$.
\item[$(2)$] $v(x) = 1$ for all $x \le 0$, and $\lim \limits _{x \to \infty} v(x) \ge 0$.
\end{itemize}
Let ${\Psi^+}$ denote the set of stochastic supersolutions.  \qed
\end{defin}

$\Psi^+$ is non-empty because $1 \in \Psi^+$.  However, it is more useful to have a stochastic supersolution that satisfies the boundary conditions with equality; therefore, we present the following lemma.

\begin{lemma}\label{lem:opsi}
Define the function $\opsi$ on $\R$ by
\begin{equation}\label{eq:opsi}
\opsi(x) = e^{- \rhoj x} \wedge 1,
\end{equation}
in which $\rhoj$ is the maximum adjustment coefficient given in Theorem {\rm \ref{thm:rhoj}}.  Then, $\opsi$ is a stochastic supersolution.
\end{lemma}

\begin{proof}
By construction, $\opsi$ is in $\mC^0(\R)$, $\opsi(x) = 1$ for all $x \le 0$, and $\lim_{x \to \infty} \opsi(x) = 0$; thus, $\opsi$ is u.s.c.\ and satisfies condition (2) in Definition \ref{def:stoch_super} with equality.  To show condition (1) of that definition, consider a random initial condition $(\tau, \zeta)$, and define the stationary retention strategy $\widetilde \mR_J = \big\{ \big(\hRj \big)_t \big\}_{\tau \le t \le \tau_0}$ by
\begin{equation*}
\big(\hRj \big)_t(y) = \hRj(y),
\end{equation*}
in which the function $\hRj$ is given in \eqref{eq:hRj}.  For $x > 0$,
\begin{align}\label{eq:cR_xpos}
&- \, \dfrac{1}{2} \, \bet^2 \, \opsi_{xx}(x) + \left( \kap - \la \left( (1 + \tet) \E \hRj + \eta \E \big(Y \hRj \big) - \dfrac{\eta}{2} \, \E \big( \hRj^2 \big) \right) \right) \opsi_x(x) \notag \\
&\quad  -  \la \left( \E \opsi \big(x - \hRj \big) - \opsi(x) \right) \ge 0.
\end{align}
Indeed, it follows from \eqref{eq:rhoj} that inequality \eqref{eq:cR_xpos} is equivalent to
\begin{equation}\label{eq:ineq3}
\int_0^\infty \left( e^{\rho \hRj(y)} - e^{\rho \big(x \wedge \hRj(y) \big)} \right)  dF_Y(y) \ge 0,
\end{equation}
which is clearly true.

By applying a general version of It\^o's formula (see Protter \cite{P2005}) to $\opsi \Big( X^{\tau, \zeta, \widetilde \mR_J}_{\omega} \Big)$, for any $\mathbb{F}$-stopping time $\omega \in [\tau, \tau_0]$, we obtain
\begin{align}\label{eq:opsi_Ito}
&\opsi \Big( X^{\tau,\zeta, \widetilde \mR_J}_{\omega} \Big) = \opsi(\zeta) + M_{\omega}
+ \bet \int_\tau^{\omega} \opsi_x \Big(X^{\tau, \zeta, \widetilde \mR_J}_t \Big) d W_t\\
&\qquad - \int_\tau^{\omega} \bigg[ - \dfrac{1}{2} \, \bet^2 \, \opsi_{xx} \Big(X^{\tau, \zeta, \widetilde \mR_J}_t \Big) -  \la \Big( \E \opsi \Big(X^{\tau,\zeta, \widetilde \mR_J}_t - \big(\hRj \big)_t \Big) - \opsi \Big(X^{\tau,\zeta, \widetilde \mR_J}_t \Big) \Big) \notag \\
&\qquad \qquad \qquad \quad ~~ + \left( \kap - \la \left( (1 + \tet) \E \big(\hRj \big)_t + \eta \E\big(Y \big(\hRj \big)_t \big) - \dfrac{\eta}{2} \, \E \big(\big(\hRj \big)^2_t \big) \right) \right) \opsi_x \Big(X^{\tau, \zeta, \widetilde \mR_J}_t \Big)  \bigg] dt,  \notag
\end{align}
in which
\begin{align*}
M_t & = \sum \limits_{\substack{ X^{\tau,\zeta, \widetilde \mR_J}_s \ne X^{\tau,\zeta, \widetilde \mR_J}_{s-}\\ \tau \le s \le t }} \left(\opsi \Big(X^{\tau,\zeta, \widetilde \mR_J}_s \Big) - \opsi \Big(X^{\tau,\zeta, \widetilde \mR_J}_{s-} \Big) \right) \\
 & \quad - \la \int_\tau^{t} \left(\E \opsi \Big(X^{\tau, \zeta, \widetilde \mR_J}_s - \big(\hRj \big)_s \Big)
- \opsi \Big(X^{\tau,\zeta, \widetilde \mR_J}_s \Big)\right) ds
\label{eq:M_t}
\end{align*}
is a martingale with zero $\mF_\tau$-expectation.  From \eqref{eq:cR_xpos}, we know that the integrand \eqref{eq:opsi_Ito} is non-negative.  Thus, by taking the $\mF_\tau$-expectation of the expression in \eqref{eq:opsi_Ito}, we obtain
\[
\opsi(\zeta) \ge \E \left[ \opsi \Big({X}^{\tau,\zeta, \widetilde \mR_J}_\omega \Big) \Big | \, \mF_{\tau}\right].
\]
Thus, $\opsi$ satisfies condition (1) in Definition \ref{def:stoch_super}, in which $\widetilde \mR_J$ is associated with $\opsi$ for any initial condition $(\tau, \zeta)$.
\end{proof}

The next proposition gives us a relationship between the value function $\psi$ and stochastic supersolutions.

\begin{prop}\label{prop:psi_v}
For any $v \in \Psi^+$, we have $\psi \le v$ on $\R$, that is, the minimum probability of ruin is a lower bound of any stochastic supersolution.
\end{prop}

\begin{proof}  First, note that $\psi(x) = v(x) = 1$ for all $x \le 0$ and $\lim_{x \to \infty} \psi(x) = 0 \le \lim_{x \to \infty} v(x)$ by condition (2) in Definition \ref{def:stoch_super}.  Second, for $x > 0$, let $(\tau, \zeta) = (0, x)$, and let $\mR$ be associated with $v$ for this initial condition.  By applying the supermartingale property (1) in Definition \ref{def:stoch_super} with $\omega = \tau_0$, the time of ruin, and by recalling that $v(\Delta) = 1$, we have
\begin{align*}
~~\qquad v(x) \ge \, \E \left[ v \Big(X_{\tau_0}^{x, \mR} \Big) \right] \,
&= \E \left[ v \Big(X_{\tau_0}^{x, \mR} \Big) {\bf 1}_{\{\tau_0 < \infty\}} \right] + \E \left[ v \Big(X_{\tau_0}^{x, \mR} \Big) {\bf 1}_{\{\tau_0 = \infty\}} \right] \notag\\
& \ge \mathbb{P}^x \left[ \tau_0 < \infty \right] \ge \psi(x),
\end{align*}
in which $\mathbb{P}^x$ denotes probability conditional on $X_0 = x$.
\end{proof}

The proof of the following result is similar to the proof of Theorem 3.1 in Liang and Young \cite{LY2019}, so we omit it.

\begin{theorem}\label{thm:v_plus}
The upper stochastic envelope $v_{+}$, defined by
\begin{equation}\label{eq:v_plus}
v_{+}(x) = \inf \limits_{v \in {\Psi^+}} v(x),
\end{equation}
for $x \in \R$, is a viscosity subsolution of \eqref{HJB_F} and \eqref{boundary_condition}.  \qed
\end{theorem}

As immediate corollary of the definition of $v_{+}$ in \eqref{eq:v_plus} and of Proposition \ref{prop:psi_v}, we have the following result.

\begin{cor}\label{cor:psi_vplus}
$\psi \le v_{+}$ on $\R$, that is, the minimum probability of ruin is a lower bound of $v_{+}$.  \qed
\end{cor}

\subsection{Stochastic subsolution}

\begin{defin} \label{def:stoch_sub}
An l.s.c.\ function $u: \R \to \left[ 0, 1 \right]$ is called a {\rm stochastic subsolution} if it satisfies the following properties:
\begin{itemize}
\item[$(1)$]
For any random initial condition $(\tau, \zeta)$, any retention strategy $\mR \in \fR_{\tau, \tau_0}$, and any $\mathbb{F}$-stopping time $\omega \in [\tau, \tau_0]$,
\[
u(\zeta) \le \E \left[ u \big({X}^{\tau, \zeta, \mR}_\omega \big) \Big | \, \mF_{\tau} \right] ~~ a.s.,
\]
in which $u$ is understood to be its extension to $\R \cup \{\Delta \}$.
\item[$(2)$] $u(x) \le 1$ for all $x \le 0$, and $\lim \limits _{x \to \infty} u(x) = 0$.
\end{itemize}
Let ${\Psi^-}$ denote the set of stochastic subsolutions.  \qed
\end{defin}

$\Psi^-$ is non-empty because $0 \in \Psi^-$.  However, it is more useful to have a stochastic subsolution that satisfies the boundary conditions with equality; therefore, we present the following lemma.

\begin{lemma}\label{lem:upsi}
Define the function $\upsi$ on $\R$ by
\begin{equation}\label{eq:upsi}
\upsi(x) = e^{-\gam x} \wedge 1,
\end{equation}
in which $\gam$ satisfies
\begin{equation}\label{eq:alp}
\gam \ge \dfrac{2c}{\bet^2} \, .
\end{equation}
Then, $\upsi$ is a stochastic subsolution.
\end{lemma}

\begin{proof}
By construction, $\upsi$ is in $\mC^0(\R)$, $\upsi(x) = 1$ for $x \le 0$, and $\lim_{x \to \infty} \upsi(x) = 0$; thus, $\upsi$ is l.s.c.\ and satisfies condition (2) in Definition \ref{def:stoch_sub} with equality.  Thus, we only need to show condition (1) of that definition.  For $x > 0$,
\[
F\left(x, \, \upsi(x), \, \upsi'(x), \, \upsi''(x), \, \upsi(\cdot)\right) \le 0,
\]
if and only if
\begin{align*}
&- \, \dfrac{1}{2} \, \bet^2 \, \upsi''(x) \\
&\quad  + \sup_R \left[ \left( \kap - \la \left( (1 + \tet) \E R + \eta \E \big(Y R \big) - \dfrac{\eta}{2} \, \E \big( R^2 \big) \right) \right) \upsi'(x) -  \la \left( \E \upsi \big(x - R \big) - \upsi(x) \right) \right] \le 0,
\end{align*}
or equivalently,
\begin{align*}
- \, \dfrac{1}{2} \, \bet^2 \, \gam^2 + \sup_R \left[ - \left( \kap - \la \left( (1 + \tet) \E R + \eta \E \big(Y R \big) - \dfrac{\eta}{2} \, \E \big( R^2 \big) \right) \right) \gam -  \la \left( \E \big( e^{\gam R} \wedge e^{\gam x} \big) - 1 \right) \right] \le 0.
\end{align*}
The last inequality holds if the following stronger one does:
\begin{align}\label{eq:alp_ineq}
- \, \dfrac{1}{2} \, \bet^2 \, \gam + \sup_R \left[ \la \left( (1 + \tet) \E R + \eta \E \big(Y R \big) - \dfrac{\eta}{2} \, \E \big( R^2 \big) \right) - \kap \right] \le 0,
\end{align}
or equivalently, because the argmax in \eqref{eq:alp_ineq} is given by $R = Y$ and because we are looking for $\gam > 0$,
\[
\dfrac{1}{2} \, \bet^2 \, \gam^2 - c \gam \ge 0,
\]
which is a restatement of \eqref{eq:alp} when $\gam > 0$.

Hence, by applying It\^o's formula to $\upsi \big(X_{\omega}^{\tau, \zeta, \mR} \big)$ with initial random condition $(\tau, \zeta)$, retention strategy $\mR \in \fR_{\tau, \tau_0}$, and $\mathbb{F}$-stopping time $\omega \in [\tau, \tau_0]$, and by taking the $\mF_\tau$-expectation as in the proof of Lemma \ref{lem:opsi}, we obtain
$$
\upsi(\zeta) \le \E \left[ \upsi \big({X}^{\tau,\zeta, \mR}_\omega \big) \Big | \mF_{\tau} \right] ~~a.s.
$$
Therefore, $ \upsi$ satisfies condition (1) in Definition \ref{def:stoch_sub}.
\end{proof}

Next, we show a relationship between the value function $\psi$ and stochastic subsolutions.  Before that, we prove a lemma, in which we adapt the proof of Lemma 2.9 in Schmidli \cite{S2008} to our model.

\begin{lemma}\label{lem:Xp}
For any $x > 0$ and any admissible reinsurance strategy $\mR = \{R_t\}_{t \ge 0}$, on the set $\{\tau_0 = \infty\}$, we have $\lim \limits_{t \to \infty} X_t^{x, \mR} = \infty$.
\end{lemma}

\begin{proof}
Let $\mR = \{ R_t \}_{t \ge 0}$ be any admissible reinsurance strategy.  For any $R = R_t$ for $t \ge 0$, let $n(R)$ denote the rate of {\it net} premium income, that is,
$$
n(R) = c - (1 + \tet) \la \E \big( Y - R \big) - \dfrac{\eta}{2} \, \la \E \big( (Y - R)^2 \big).
$$
Also, let $A$ denote the set of possible retention functions $R = R_t$ such that  $n(R) \ge - \kap/2$.  Note that $- \kap$ is the rate of net premium income if the insurer purchases full reinsurance, that is, $c(0) = -\kap < 0$; see \eqref{eq:kap}.   Choose $0 < \eps < \kap/2$.  Let $\varsigma \in (0, 1)$ denote the constant $(\kap - 2 \eps) / (\kap + 2c)$. First, suppose $\int^{t+1}_{t} {\bf 1}_{\{R_s \in A\}} ds \le \varsigma$; then, because
\[
dX_t^{x, \mR} = n(R_t) dt - R_t dN_t + \bet dW_t,
\]
it follows that
\begin{align*}
X_{t+1}^{x, \mR} &\le X_t^{x, \mR} + \int_t^{t+1} n(R_s) ds + \bet (W_{t+1} - W_t) \\
&= X_t^{x, \mR} + \int_t^{t+1} n(R_s) {\bf 1}_{\{R_s \in A\}} ds + \int_t^{t+1} n(R_s) {\bf 1}_{\{R_s \notin A\}} ds  + \bet (W_{t+1} - W_t) \\
&\le X_t^{x, \mR} + \varsigma c - (1-\varsigma) \kap/2 + \bet (W_{t+1} - W_t) \\
&= X_t^{x, \mR} - \eps + \bet (W_{t+1} - W_t).
\end{align*}
Because $\mathbb{P} \big[W_{t+1} - W_t \le 0 \big] = 1/2$, there exists a $0 < \del <1/2$, such that
\begin{equation}\label{eq:prob1}
\mathbb{P}\Big[X_{t+1}^{x, \mR} -  X_t^{x, \mR} \le - \eps \Big] \ge \del.
\end{equation}

Second, suppose $\int^{t+1}_{t} {\bf 1}_{\{R_s \in A\}} ds > \varsigma$.  By the definition of $A$, we can assume that $\eps \in (0, \kap/2)$ from the first case is small enough so that
\begin{align}
\label{eq:Y_small}
\mathbb{P} \left[ \inf \limits_{R_t \in A} R_t > \eps \right] > 0.
\end{align}
Moreover, note that
\begin{align}
\label{eq:N_A}
\mathbb{P} \left[\int^{t+1}_{t} {\bf 1}_{\{R_s \in A\}} d N_s \ge 1 + \dfrac{c}{\eps} \, \right] \ge \mathbb{P} \left[N_{\varsigma} \ge 1 + \dfrac{c}{\eps} \, \right] > 0.
\end{align}
Thus, by combining \eqref{eq:Y_small} and \eqref{eq:N_A}, we deduce
\begin{equation}\label{eq:prob2}
\mathbb{P} \left[\sum \limits_{i=N_t+1}^{N_{t+1}} R_s(Y_i) \ge c + \eps\right] > 0,
\end{equation}
in which $s \in (t, t+1]$ in the subscript of the summand $R_s(Y_i)$ indicates the time of the $i$th claim.  Now,
\begin{align*}
X_{t+1}^{x, \mR} &= X_t^{x, \mR} + \int_t^{t+1} n(R_s) ds - \sum \limits_{i=N_t+1}^{N_{t+1}} R_s(Y_i) + \bet (W_{t+1} - W_t) \\
&\le X_t^{x, \mR} + c - \sum \limits_{i=N_t+1}^{N_{t+1}} R_s(Y_i) + \bet (W_{t+1} - W_t).
\end{align*}
As in the first case, because $\mathbb{P} \big[W_{t+1} - W_t \le 0 \big] = 1/2$, there exists a $0 < \del < 1/2$, such that
\begin{equation}\label{eq:prob3}
\mathbb{P}\Big[X_{t+1}^{x, \mR} -  X_t^{x, \mR} \le - \eps \Big] \ge \del.
\end{equation}
We can choose $\del$ to be the same lower bound in both \eqref{eq:prob1} and \eqref{eq:prob3} by taking the minimum of the two.

Next, choose $a > 0$.  Define a sequence of hitting times $\{t_k \}_{k = 0, 1, \dots}$ recursively as follows:
\[
t_0 = \inf \left\{t\ge 0: X_t^{x,\mR} \le a \right\},
\]
and, for $k = 1, 2, \dots$,
\[
t_{k+1} = \inf \left\{t\ge t_k+1: X_t^{x, \mR} \le a \right\},
\]
with the understanding that $\inf \emptyset = \infty$.  From \eqref{eq:prob1} and \eqref{eq:prob2}, we deduce
\begin{align}
\label{eq:submartingale}
\mathbb{P} \Big[X_{t_k + 1}^{x, \mR} \le a - \eps \, \Big| \, \mF_{t_k} \Big] \ge \del,
\end{align}
if $t_k < \infty$.  Let $V_k = {\bf 1}_{\left\{t_k < \infty,\, X_{t_k + 1}^{x, \mR} \le a - \eps\right\}}$ and $Z_k = \del {\bf 1}_{\{t_k < \infty\}}$; \eqref{eq:submartingale} implies $\Big\{ \sum \limits_{k=1}^{n} (V_k - Z_k) \Big\}_{n \in \N}$ is a sub-martingale.

If $\liminf_{t \to \infty} X_t^{x, \mR} \le b$ for some $b < \infty$, and if we set $a = b + \eps/2$, then we have $t_0 < \infty$.  Furthermore, $\liminf_{t \to \infty} X_t^{x, \mR} \le b$ implies $t_k < \infty$  for $k = 1, 2, \dots$ recursively, which implies $\sum \limits_{k=1}^{\infty} Z_k = \infty$.  From Lemma 1.15 in Schmidli \cite{S2008}, it follows that $\sum \limits_{k=1}^{\infty} V_k = \infty$.  Thus, $X_{t_k + 1}^{x, \mR} \le b - \eps/2$ infinitely often; hence, we have $\liminf_{t \to \infty} X_t^{x, \mR} \le b - \eps/2$.  By iterating above procedure, we deduce that $\liminf_{t \to \infty} X_t^{x, \mR} < \infty$ implies $\liminf_{t \to \infty} X_t^{x, \mR} < 0$; in other words, $\tau_0 < \infty$ if $\liminf_{t \to \infty} X_t^{x, \mR}$ is finite.  Our proof is complete.
\end{proof}

\begin{prop}\label{prop:u_psi}
For any $u \in \Psi^-$, we have $u \le \psi$ on $\R$, that is, the minimum probability of ruin is an upper bound of any stochastic subsolution.
\end{prop}

\begin{proof}
First, note that $\psi(x) = 1 \ge u(x)$ for all $x \le 0$ and $\lim_{x \to \infty} \psi(x) = 0 = \lim_{x \to \infty} u(x)$ by condition (2) in Definition \ref{def:stoch_sub}.  Second, for $x > 0$, let $(\tau, \zeta) = (0, x)$, let $\mR$ be any admissible retention strategy, and let $\omega = \tau_0$.  Then, by applying the submartingale property (1) in Definition \ref{def:stoch_sub}, we have
\begin{equation}\label{eq:u_ineq}
u(x) \le \E \Big[ u \big(X_{\tau_0}^{x, \mR}\big) \Big].
\end{equation}
Lemma \ref{lem:Xp} shows us that either ruin occurs or $\lim_{t \to \infty} X_t^{x, \mR} = \infty$; thus, because $\lim_{x \to \infty} u(x) = 0$ and $u \big(X_{\tau_0}^{x, \mR} \big) {\bf 1}_{\{\tau_0 < \infty\}} = {\bf 1}_{\{\tau_0 < \infty\}}$, inequality \eqref{eq:u_ineq} implies
\[
u(x) \le \E \Big[ u \big(X_{\tau_0}^{x, \mR}\big)  {\bf 1}_{\{\tau_0 < \infty\}} \Big] = \mathbb{P}^x \big( \tau_0 < \infty \big).
\]
Because this inequality holds for any retention strategy, by taking the infimum over all admissible retention strategies, we obtain $u \le \psi$.
\end{proof}

The proof of the following theorem is similar to the proof of Theorem 3.2 in Liang and Young \cite{LY2019}, so we omit it.

\begin{theorem}\label{thm:u_minus}
The lower stochastic envelope $u_{-},$ defined by
\begin{equation}\label{eq:u_minus}
u_{-}(x) = \sup \limits_{u \in {\Psi^-}} u(x),
\end{equation}
is a viscosity supersolution of \eqref{HJB_F} and \eqref{boundary_condition}.  \qed
\end{theorem}

As an immediate corollary of the definition of $u_{-}$ in \eqref{eq:u_minus} and of Proposition \ref{prop:u_psi}, we have the following result.

\begin{cor}\label{cor:uminus_psi}
$u_{-} \le \psi$ on $\R$, that is, the minimum probability of ruin is an upper bound of $u_{-}$.  \qed
\end{cor}

\subsection{Comparison principle}

We use Definition \ref{def:sup/subjet} to prove a comparison principle.  First, we introduce a function that we will use in that proof.

\begin{lemma}\label{lem:qm}
For a fixed value of $b > 1$, define the function $q \in \mC^2(\R)$ by
\begin{equation}\label{eq:q}
q(x) =
\begin{cases}
0, &\quad x \le 1, \\
2b \big( 6(x - 1)^5 - 15(x - 1)^4 + 10(x - 1)^3 \big), &\quad 1 < x < 2, \\
2b, &\quad x \ge 2.
\end{cases}
\end{equation}
Then, for $m \in \N$, define the function $\qm$ by
\begin{equation}\label{eq:qm}
\qm(x) = q(x/m),
\end{equation}
for all $x \in \R$.  Then, $\qm$ is non-decreasing on $\R$,
\begin{equation}\label{eq:qm_deriv_lim}
\lim_{m \to \infty} \big|\big| \qm' \big|\big|_\infty = 0 = \lim_{m \to \infty} \big|\big| \qm'' \big|\big|_\infty,
\end{equation}
and
\begin{equation}\label{eq:maxR_qm}
\sup_R \, \E \qm(x - R) - \qm(x) = 0,
\end{equation}
for all $x > 0$.
\end{lemma}

\begin{proof}
To prove the limits in \eqref{eq:qm_deriv_lim}, note that
\[
\qm'(x) = \dfrac{1}{m} \, q'(x/m) = \frac{60b}{m}\left(\frac{x}{m} - 1\right)^2\left(\frac{x}{m} - 2\right)^2 \ge 0 \,,
\]
and the maximum of $q'$ occurs at $x = 3/2$, so
\[
\big|\big| \qm' \big|\big|_\infty = \dfrac{1}{m} \, q'(3/2) = \dfrac{15b}{4m} \, .
\]
Also,
\[
\qm''(x) = \dfrac{1}{m^2} \, q''(x/m),
\]
and the maximum of $\big| q'' \big|$ occurs at $x = (9 \pm \sqrt{3} \, )/6$, so
\[
\big|\big| \qm'' \big|\big|_\infty = \dfrac{1}{m^2} \, \big| q''\big((9 \pm \sqrt{3} \, )/6 \big) \big| = \dfrac{20 b \sqrt{3}}{3 m^2} \, .
\]
The limits in \eqref{eq:qm_deriv_lim} follow easily.  Finally, \eqref{eq:maxR_qm} follows from the fact that $\qm$ is non-decreasing, so the supremum of $0$ is attained at $R = 0$.
\end{proof}

Now, we are ready to prove a comparison theorem.  We use a technique in the proof of Theorem 5.9 of Section II.5.3 of Bardi and Capuzzo-Dolcetta \cite{BC-D1997} adapted to our problem.

\begin{theorem}\label{thm:comp} {\rm (Comparison principle)}
If $v ~(\text{resp.}, u)$ is a viscosity subsolution $($resp., viscosity supersolution$)$ of \eqref{HJB_F} and \eqref{boundary_condition}, then $v \le u$ on $\R$.
\end{theorem}

\begin{proof}
For $t \in (0,1)$, define $v^t$ on $\R$ by
$$
v^{t}(x) = t v(x) + (1-t) \upsi(x),
$$
in which $\upsi$ is given by \eqref{eq:upsi} for some $\gam > 2c/\bet^2$.  It is easy to see that $v^{t}$ is u.s.c.\ and $v^{t} \to v$ as $t \to 1$.  From the proof of Lemma \ref{lem:upsi}, we know that, for all $x > 0$, $\upsi$ is a strict subsolution of $F = 0$, that is,
\begin{align}\label{oper:F0_var}
F \big(x, \upsi(x), \upsi'(x), \upsi''(x), \upsi(\cdot) \big) < 0.
\end{align}
Because $v$ is a u.s.c.\ viscosity subsolution, from Definition \ref{def:sup/subjet}, it follows that, for all $x > 0$ and $(p, X) \in J^{2,+}v(x)$,
$$
F \big(x, v(x), p, X, v(\cdot) \big) \le 0.
$$
Furthermore, by the definition of superjets in Definition \ref{def:sup/subjet1}, we have
$$
J^{2,+} v^{t}(x) = \big\{(q,Y) \in \R^2:\, q = tp + (1 - t) \upsi'(x),\, Y = tX + (1 - t)\upsi''(x),\, (p, X) \in J^{2,+} v(x) \big\} ,
$$
for $x > 0$.  Thus, for all $x > 0$ and $(q, Y) = \big(tp + (1 - t) \upsi'(x), \, tX + (1 - t)\upsi''(x) \big) \in J^{2,+} v^t(x)$, with $(p, X) \in J^{2,+}v(x)$, we have
\begin{align*}
& F\big(x, v^t(x),q, Y, v^t(\cdot) \big) \\
&= F\big(x, v^t(x),tp + (1 - t)\upsi'(x), tX + (1 - t)\upsi''(x), v^t(\cdot) \big) \\
& = - \, \frac{1}{2} \, \bet^2 \big(tX + (1 - t) \upsi''(x) \big) + \la \sup_{R}\bigg\{ - \big( \E v^t(x - R) - v^t(x)\big) \\
& \qquad \qquad \qquad \qquad \qquad \qquad \qquad   - \big(t p + (1 - t) \upsi'(x) \big) \left( (1 + \tet) \E R + \eta \E(YR)
- \dfrac{\eta}{2} \, \E \big(R^2 \big) - \frac{\kap}{\la} \right) \bigg\} \\
& \le -\, \frac{t}{2} \,\bet^2 X + \la t \, \sup_{R} \bigg\{ - p \left( (1 + \tet) \E R + \eta \E(YR)
- \dfrac{\eta}{2} \, \E \big(R^2 \big) - \frac{\kap}{\la} \right) - \big(\E v(x - R) - v(x)\big) \bigg\} \\
& \quad - \, \frac{1-t}{2} \, \bet^2 \upsi''(x) + \la (1-t) \sup_{R}\bigg\{- \upsi'(x) \left( (1 + \tet) \E R + \eta \E(YR)- \dfrac{\eta}{2} \, \E \big(R^2 \big) - \frac{\kap}{\la} \right)\\
&\qquad \qquad \qquad \qquad \qquad \qquad \qquad \qquad -\left(\E \upsi(x - R) - \upsi(x)\right) \bigg\} \\
& = t F \big( x, v(x), p, X, v(\cdot) \big) + (1- t) F \big(x, \upsi(x), \upsi'(x), \upsi''(x), \upsi(\cdot) \big) \\
&\le (1-t) F \big(x, \upsi(x), \upsi'(x), \upsi''(x), \upsi(\cdot) \big),
\end{align*}
which implies that, for all $t \in (0,1)$, $v^t$ is a viscosity subsolution of
\begin{align}
F\big(x, u(x), u_x(x), u_{xx}(x), u(\cdot) \big) - (1 - t) g(x) = 0,
\label{oper_vt}
\end{align}
for all $x > 0$, in which $g$ is defined by
\begin{equation}
g(x) = F \big(x, \upsi(x), \upsi'(x), \upsi''(x), \upsi(\cdot) \big).
\end{equation}

Recall that $v^t(x) = v(x) = u(x) = 1$ for all $x \le 0$.  We wish to prove that for all $t \in (0, 1)$ and $x > 0$, $v^t(x) \le u(x)$.  By letting $t$ go to $1$, it will, then, follow that $v(x) \le u(x)$ for all $x > 0$.  Suppose, on the contrary, for some $t^* \in (0,1)$,
\begin{equation}\label{eq:S}
S := \sup_{x \in \R^+} \big( \vs(x) - u(x) \big) > 0,
\end{equation}
in which, for simplicity, we write $\vs$ instead of $v^{t^*}$.  Note that $S$ is finite because $u$ and $\vs$ are bounded.  We, next, approximate $S$.  To that end, define $q$ and $\qm$ by \eqref{eq:q} and \eqref{eq:qm}, respectively, for some $b > 1$.
Define $\Smn$, which we use to approximate $S$, as follows:
\begin{equation}\label{eq:Smn}
\Smn = \sup_{x, y \in \R^+} \left(\vs(x) - u(y) - \qm(x) - \dfrac{n}{2} \, (x - y)^2 \right).
\end{equation}
Because $S > 0$, there exists $x' > 0$ such that
\[
\vs(x') - u(x') \ge \dfrac{S}{2} \, .
\]
For the remainder of this proof, assume $m > x'$; then,
\begin{align}\label{eq:approx}
\Smn &= \sup_{x, y \in \R^+} \left(\vs(x) - u(y) - \qm(x) - \dfrac{n}{2} \, (x - y)^2 \right) \notag \\
&\ge \sup_{x \in \R^+} \left(\vs(x) - u(x) - \qm(x) - \dfrac{n}{2} \, (x - x)^2 \right) \notag \\
&= \sup_{x \in \R^+} \big(\vs(x) - u(x) - \qm(x) \big) \notag \\
&\ge \sup_{0 \le x \le m} \big(\vs(x) - u(x) - \qm(x) \big) \notag \\
&= \sup_{0 \le x \le m} \big(\vs(x) - u(x) \big) \notag \\
&\ge \vs(x') - u(x') \ge \dfrac{S}{2} \, .
\end{align}
Among other things, we have $\Smn > 0$, which implies that the supremum defining $\Smn$ is achieved on $[0, 2m]^2$ because $\qm(x) > ||u||_\infty + ||\vs||_\infty$ for $x \ge 2m$.  Let $(\xmn, \ymn)$ be a point at which the supremum $\Smn$ is realized.  For a fixed value of $m > x'$, the sequence $\{(\xmn, \ymn)\}_{n \ge N}$ lies in a bounded region, namely $[0, 2m]^2$, which implies that this sequence converges to some point $(x_{m, \infty}, y_{m, \infty}) \in [0, 2m]^2$ as $n$ goes to $\infty$.  Furthermore, the inequality
\begin{equation}\label{eq:approx2}
\vs(\xmn) - u(\ymn) - \qm(\xmn) - \dfrac{n}{2} \, (\xmn - \ymn)^2 \ge \dfrac{S}{2}
\end{equation}
holds for all $n \in \N$; thus, there exist $C > 0$ and $N \in \N$ such that, for all $n \ge N$,
\begin{equation}\label{eq:diff_bnd}
n (\xmn - \ymn)^2 \le C,
\end{equation}
which implies that $x_{m, \infty} = y_{m, \infty}$.

We can obtain even more from the inequalities in \eqref{eq:approx}.  First, note that
\begin{equation}\label{eq:m_infty}
\lim_{m \to \infty} \sup_{0 \le x \le m} \big( \vs(x) - u(x) \big) = \sup_{x \in \R^+} \big( \vs(x) - u(x) \big),
\end{equation}
in which the right side equals $S$.  Indeed, because $\vs - u$ is u.s.c., it follows that, on the set $0 \le x \le m$, $\vs - u$ achieves its supremum at, say, $\hat x_m$.  Then, because the interval $[0, m]$ increases with $m$, the sequence $\{ \vs(\hat x_m) - u(\hat x_m) \}_{m > x'}$ is non-decreasing.  Also, this sequence is bounded above by $S$; therefore, it has a limit $S'$.  Clearly, $S' \le S$, and we wish to show that $S' = S$.  Suppose, on the contrary, that $S' < S$, and define $\varsigma = (S - S')/2$.  By the definition of $S$, there exists $\tilde x$ such that
\[
\vs(\tilde x) - u(\tilde x) > S - \varsigma = \dfrac{S + S'}{2} > S',
\]
which contradicts the definition of $S'$.  Thus, $S' = S$.

Now, because $\qm \ge 0$, from inequality \eqref{eq:approx}, we have
\[
\sup_{x \in \R^+} \big( \vs(x) - u(x) \big) \ge \sup_{x \in \R^+} \big(\vs(x) - u(x) - \qm(x) \big) \ge \sup_{0 \le x \le m} \big(\vs(x) - u(x) \big),
\]
or equivalently,
\begin{equation}\label{eq:sub_approx}
S \ge \Sm \ge \sup_{0 \le x \le m} \big(\vs(x) - u(x) \big),
\end{equation}
in which $\Sm$ denotes the supremum of $\vs  - u - \qm$ on $\R^+$.  Then, by taking the limit as $m$ goes to $\infty$ in \eqref{eq:sub_approx} and by using \eqref{eq:m_infty}, we obtain
\[
S \ge \limsup_{m \to \infty} \Sm \ge \liminf_{m \to \infty} \Sm \ge S,
\]
which implies that
\begin{equation}\label{eq:approx_lim}
\lim_{m \to \infty} S_m = S.
\end{equation}
Also, inequality \eqref{eq:approx} implies that $\vs(\xmn) - u(\ymn) - \qm(\xmn) \ge \Smn \ge \Sm$ for all $n \in \N$; now, let $n$ go to $\infty$ to obtain
\[
S \ge \vs(x_{m, \infty}) - u(x_{m, \infty}) \ge \limsup_{n \to \infty} \big( \vs(\xmn) - u(\ymn) - \qm(\xmn) \big) \ge \limsup_{n \to \infty} \Smn \ge \Sm,
\]
in which the second inequality follows because $\vs - u$ is u.s.c.  Thus, we have
\begin{equation}\label{eq:approx_lim2}
\lim_{m \to \infty} \limsup_{n \to \infty} \Smn = \lim_{m \to \infty} \Sm = S = \lim \limits_{m \to \infty} \big( \vs(x_{m, \infty}) - u(x_{m, \infty}) \big),
\end{equation}
and $\lim \limits_{m \to \infty} \qm(x_{m, \infty}) = 0$.
Similarly,
\begin{align*}
\Sm &\ge \vs(x_{m, \infty}) - u(x_{m, \infty}) - \qm(x_{m, \infty}) \\
&\ge \limsup_{n \to \infty} \Big( \vs(\xmn) - u(\ymn) - \qm(\xmn) - \dfrac{n}{2} \, (\xmn - \ymn)^2 \Big) \\
&= \limsup_{n \to \infty} \Smn \ge \Sm,
\end{align*}
which implies that
\begin{equation}\label{eq:approx_lim3}
\limsup_{n \to \infty} \Smn = \Sm = \vs(x_{m, \infty}) - u(x_{m, \infty}) - \qm(x_{m, \infty}),
\end{equation}
and $\lim \limits_{n \to \infty} n (\xmn - \ymn)^2 = 0$.

From the definition of $(\xmn, \ymn)$, we deduce that
\begin{equation}
\begin{cases}
\xmn ~ \hbox{is a maximizer of} ~ x \mapsto \vs(x) - \qm(x) - \dfrac{n}{2} \, (x - \ymn)^2, \\
\ymn ~ \hbox{is a minimizer of} ~ y \mapsto u(y) + \dfrac{n}{2} \, (\xmn - y)^2.
\end{cases}
\end{equation}
By using Lemma 1 (a non-local Jensen-Ishii's lemma) and Corollary 1 in Barles and Imbert \cite{BI2008}, we know there exists a sufficiently large constant $K$ such that, for each $n \ge K$ and each $m > x'$, there exists $(\Amn, \Bmn) \in \R^2$, with
$$
\big(\qm'(\xmn) + n(\xmn - \ymn), \Amn \big) \in \bar{J}^{2,+} {\vs}(\xmn),
$$
$$
\big( n(\xmn - \ymn), \Bmn \big) \in \bar{J}^{2,-} {u}(\ymn),
$$
and
\begin{align} \label{eq:matrixy}
-n \left(
                      \begin{array}{cc}
                       1 & 0 \\
                        0 & 1 \\
                      \end{array}
                    \right)
 \le \left(
                      \begin{array}{cc}
                       \Amn & 0 \\
                        0 & - \Bmn \\
                      \end{array}
                    \right)
         \le \left(
                         \begin{array}{cc}
                             \qm''(\xmn) + n & -n \\
                              -n & n \\
                           \end{array}
                          \right) + o \big(n^{-1} \big).
\end{align}
From \eqref{eq:matrixy}, one can easily derive
\begin{align} \label{eq:X_minus_Y}
\Amn - \Bmn \le \qm''(\xmn) + o \big(n^{-1} \big).
\end{align}
Also, from Definition \ref{def:sup/subjet} and from \eqref{oper_vt}, by using the viscosity subsolution property of $\vs$, we have
\begin{align*}
F \big(\xmn, \vs(\xmn), \qm'(\xmn) + n(\xmn - \ymn), \Amn, \vs(\cdot) \big) - (1 - t^*)g(\xmn) \le 0,
\end{align*}
or equivalently,
\begin{align}
\label{eq:sub_asy}
&- \, \dfrac{1}{2} \, \bet^2 \Amn + \kap \big( \qm'(\xmn) + n(\xmn - \ymn) \big) \notag \\
&- \la \inf_R \bigg[ \big( \qm'(\xmn) + n(\xmn - \ymn) \big) \left( (1 + \tet) \E R + \eta \E(YR) - \dfrac{\eta}{2} \, \E \big(R^2 \big) \right) \notag \\
&\qquad \qquad + \E \vs(\xmn - R) - \vs(\xmn) \bigg] \le (1 - t^*)g(\xmn).
\end{align}
Similarly, by using the viscosity supersolution property of $u$, we have
\begin{align*}
F \big(\ymn, u(\ymn), n (\xmn - \ymn), \Bmn, u(\cdot) \big) \ge 0,
\end{align*}
or equivalently,
\begin{align}
\label{eq:sup_asy}
&- \, \dfrac{1}{2} \, \bet^2 \Bmn + \kap n(\xmn - \ymn) \notag \\
&- \la \inf_R \left[ n(\xmn - \ymn)\left( (1 + \tet) \E R + \eta \E(YR) - \dfrac{\eta}{2} \, \E \big(R^2 \big) \right) + \E u(\ymn - R) - u(\ymn) \right] \ge 0.
\end{align}
By subtracting inequality \eqref{eq:sup_asy} from \eqref{eq:sub_asy}, we obtain the inequality
\begin{align}\label{eq:difference}
&\la \big(\vs(\xmn) - u(\ymn) \big) - \dfrac{1}{2} \, \bet^2 \big( \Amn - \Bmn \big) + \kap \qm'(\xmn) \notag \\
&- \la \inf_R \left[ \big( \qm'(\xmn) + n(\xmn - \ymn) \big) \left( (1 + \tet) \E R + \eta \E(YR) - \dfrac{\eta}{2} \, \E \big(R^2 \big) \right)  + \E \vs(\xmn - R) \right] \notag \\
&+ \la \inf_R \left[ n (\xmn - \ymn) \left( (1 + \tet) \E R + \eta \E(YR) - \dfrac{\eta}{2} \, \E \big(R^2 \big) \right) + \E u(\ymn - R) \right] \le (1 - t^*)g(\xmn).
\end{align}
Note that
\begin{align*}
&\la \inf_R \left[  \E u(\ymn - R) - \E \vs(\xmn - R) - \qm'(\xmn)  \left( (1 + \tet) \E R + \eta \E(YR) - \dfrac{\eta}{2} \, \E \big(R^2 \big) \right) \right] \\
&\le \la \inf_R \left[ n (\xmn - \ymn) \left( (1 + \tet) \E R + \eta \E(YR) - \dfrac{\eta}{2} \, \E \big(R^2 \big) \right)  + \E u(\ymn - R) \right] \\
&\quad - \la \inf_R \left[ \big( \qm'(\xmn) + n(\xmn - \ymn) \big) \left( (1 + \tet) \E R + \eta \E(YR) - \dfrac{\eta}{2} \, \E \big(R^2 \big) \right) + \E \vs(\xmn - R) \right].
\end{align*}
The above inequality, \eqref{eq:X_minus_Y} and \eqref{eq:difference} imply
\begin{align*}
&\la \big(\vs(\xmn) - u(\ymn) \big) + \kap \qm'(\xmn) - \frac{1}{2} \, \bet^2 \qm''(\xmn) + o\big(n^{-1} \big) \\
&\quad + \la \inf_R \left[  \E u(\ymn - R) - \E \vs(\xmn - R) - \qm'(\xmn)  \left( (1 + \tet) \E R + \eta \E(YR) - \dfrac{\eta}{2} \, \E \big(R^2 \big) \right) \right] \\
& \le (1 - t^*)g(\xmn),
\end{align*}
or equivalently,
\begin{align}\label{eq:equiv}
&\la \big(\vs(\xmn) - u(\ymn) \big) + \kap \qm'(\xmn) - \frac{1}{2} \, \bet^2 \qm''(\xmn) + o\big(n^{-1} \big) \notag \\
&\le \la \sup_R \left[\E \vs(\xmn - R) - \E u(\ymn - R) + \qm'(\xmn)  \left( (1 + \tet) \E R + \eta \E(YR) - \dfrac{\eta}{2} \, \E \big(R^2 \big) \right) \right] \notag \\
&\quad + (1 - t^*)g(\xmn).
\end{align}
Also, note that
\begin{align*}
&\la \sup_R \left[\E \vs(\xmn - R) - \E u(\ymn - R) + \qm'(\xmn)  \left( (1 + \tet) \E R + \eta \E(YR) - \dfrac{\eta}{2} \, \E \big(R^2 \big) \right) \right] \\
&\le \la \sup_R \E \big[ \vs(\xmn - R) - u(\ymn - R) - \qm(\xmn - R) \big]  \\
&\quad + \la \sup_R \left[ \E \qm(\xmn - R) + \qm'(\xmn)  \left( (1 + \tet) \E R + \eta \E(YR) - \dfrac{\eta}{2} \, \E \big(R^2 \big) \right) \right].
\end{align*}
The above inequality and \eqref{eq:equiv} imply
\begin{align}\label{eq:equiv2}
&\la \big(\vs(\xmn) - u(\ymn) - q_m(\xmn) \big) - \frac{1}{2} \, \bet^2 \qm''(\xmn) + o\big(n^{-1} \big) \notag \\
&\le \la \sup_R \E \bigg[ \vs(\xmn - R) - u(\ymn - R) - \qm(\xmn - R) \bigg] \notag \\ 
&\quad + \la \sup_R \left[ \E \qm(\xmn - R) - q_m(\xmn) + \qm'(\xmn)  \left( (1 + \tet) \E R + \eta \E(YR) - \dfrac{\eta}{2} \, \E \big(R^2 \big) - \dfrac{\kap}{\la} \right) \right] \notag \\
&\quad  + (1 - t^*)g(\xmn).
\end{align}
From \eqref{oper:F0_var}, there exist $N \in \N$ and $a > 0$ such that if $m \ge N$ and $n \ge N$, then
\begin{equation}\label{eq:g_bound}
g(\xmn) = F \big(\xmn, \upsi(\xmn), \upsi'(\xmn), \upsi''(\xmn), \upsi(\cdot) \big) \le -a.
\end{equation}
Moreover, from \eqref{eq:approx_lim3} and $\lim \limits_{n \to \infty} n (\xmn - \ymn)^2 = 0$, we have
\begin{align}\label{eq:equilim}
\limsup_{n \to \infty} \big(\vs(\xmn) - u(\ymn) - q_m(\xmn) \big) = \Sm = \vs(x_{m, \infty}) - u(x_{m, \infty}) - \qm(x_{m, \infty}).
\end{align}
Because $\lim \limits_{n \to \infty} \sup_R f(R, n) \le \sup_R \lim \limits_{n \to \infty} f(R, n)$, if we take a limit as $n \to \infty$ with $m \ge N$, then inequalities \eqref{eq:equiv2} and \eqref{eq:g_bound} and equality \eqref{eq:equilim} imply
\begin{align*}
&\la \big(\vs(x_{m, \infty}) - u(x_{m, \infty}) - \qm(x_{m, \infty}) \big) - \frac{1}{2} \, \bet^2 \qm''(x_{m, \infty}) \notag \\
&\le \la \sup_R \E \big[ \vs(x_{m, \infty} - R) - u(x_{m, \infty} - R) - \qm(x_{m, \infty} - R) \big] \notag \\
&\quad + \la \left( \sup_R \E \qm(x_{m, \infty} - R) - \qm(x_{m, \infty})  + || \qm' ||_\infty \left\{ (1 + \tet) \E Y + \dfrac{\eta}{2} \, \E \big( Y^2 \big) - \dfrac{\kap}{\la} \right\} \right) \notag \\
&\quad  - a(1 - t^*).
\end{align*}
Now,
$$\sup_R \E \big[ \vs(x_{m, \infty} - R) - u(x_{m, \infty} - R) - \qm(x_{m, \infty} - R) \big]\le \sup\limits_{x \in \R^+} ( \vs(x) - u(x) - \qm(x)) = \Sm;$$
 thus, we have
\begin{equation}\label{ineq3}
0 \le \la \Big| \Big| \sup_R \, \E \qm(\cdot - R) - \qm(\cdot) \Big| \Big|_\infty + c \big|\big| \qm' \big|\big|_\infty + \frac{1}{2} \, \bet^2 \big|\big| \qm'' \big|\big|_\infty - a(1 - t^*).
\end{equation}
By taking a limit as $m$ goes to $\infty$ in \eqref{ineq3} and by using the results of Lemma \ref{lem:qm}, we obtain
\begin{equation}\label{eq:lim_contra}
0 \le -a(1 - t^*),
\end{equation}
a contradiction because $a > 0$ and $t^* \in (0, 1)$.  Thus, $S \le 0$, and we have shown that $v^t(x) \le u(x)$ for all $x > 0$ and all $t \in (0, 1)$, from which we deduce that $v \le u$ on $\R$.
\end{proof}

We now present the main goal of this appendix, an application of the comparison principle in Theorem \ref{thm:comp}.

\begin{theorem}\label{thm:value}
The minimum probability of ruin $\psi$ is the unique $($continuous$)$ viscosity solution of the HJB equation \eqref{HJB_F} with boundary conditions \eqref{boundary_condition}.
\end{theorem}

\begin{proof}
From Corollaries \ref{cor:psi_vplus} and \ref{cor:uminus_psi}, we know
\begin{equation}\label{eq:uminus_psi_vplus}
u_{-} \le \psi \le v_{+}
\end{equation}
on $\R$.  Furthermore, Theorems \ref{thm:v_plus} and \ref{thm:u_minus} prove that $v_{+}$ and $u_{-}$ are viscosity sub- and supersolutions, respectively.  Thus, Theorem \ref{thm:comp} implies that $v_{+} \le u_{-}$, which, when combined with \eqref{eq:uminus_psi_vplus} implies that
\[
u_{-} = \psi = v_{+}
\]
on $\R$.  Thus, we have proved this theorem.
\end{proof}

\end{document}